\documentclass{amsart}
\usepackage{amssymb}
\usepackage{amsfonts}
\usepackage{amsmath}
\usepackage{enumitem}
\newcommand{\AFF}{\operatorname{aff}}
\newcommand{\BD}{\operatorname{bd}}
\newcommand{\BEXP}{\operatorname{b-exp}}
\newcommand{\BH}{\operatorname{bh}_1}
\newcommand{\CL}{\operatorname{cl}}
\newcommand{\CONV}{\operatorname{conv}}
\newcommand{\DIAM}{\operatorname{diam}}
\newcommand{\DIM}{\operatorname{dim}}
\newcommand{\DIST}{\operatorname{dist}}
\newcommand{\INT}{\operatorname{int}}
\newcommand{\RAD}{\operatorname{rad}}
\newcommand{\cK}{{\mathcal K}}
\newcommand{\bR}{{\mathbb{R}}}

\theoremstyle{plain}
\newtheorem{corollary}{Corollary}
\newtheorem{lemma}{Lemma}
\newtheorem{proposition}{Proposition}
\newtheorem{theorem}{Theorem}
\theoremstyle{definition}
\newtheorem{example}{Example}
\newtheorem{remark}{Remark}

\begin{document}


\title{BALL CONVEX BODIES IN MINKOWSKI SPACES}

\author{Thomas Jahn}
\address{Faculty of Mathematics, University of Technology, 09107 Chemnitz, GERMANY}
\email{thomas.jahn@mathematik.tu-chemnitz.de}

\author{Horst Martini}
\address{Faculty of Mathematics, University of Technology, 09107 Chemnitz, GERMANY}
\email{martini@mathematik.tu-chemnitz.de}

\author{Christian Richter}
\address{Institute of Mathematics, Friedrich Schiller University, 07737 Jena, GERMANY}
\email{christian.richter@uni-jena.de}

\dedicatory{To our teachers, colleagues and friends Prof.\ Dr.\ Johannes B\"ohm, on the occasion of his 90th birthday, and Prof.\ Dr.\ Eike Hertel, on the occasion of his 75th birthday.}

\date{\today}

\begin{abstract}
The notion of ball convexity, considered in finite dimensional real Banach spaces, is a natural and useful extension of usual convexity; one replaces intersections of half-spaces by suitable intersections of balls. A subset $S$ of a normed space is called ball convex if it coincides with its ball hull, which is obtained as intersection of all balls (of fixed radius) containing $S$. Ball convex sets are closely related to notions like ball polytopes, complete sets, bodies of constant width, and spindle convexity. We will study geometric properties of ball convex bodies in normed spaces, for example deriving separation theorems, characterizations of strictly convex norms, and an application to complete sets. Our main results refer to minimal representations of ball convex bodies in terms of their ball exposed faces, to representations of ball hulls of sets via unions of ball hulls of finite subsets, and to ball convexity of increasing unions of ball convex bodies.
\end{abstract}


\maketitle


\section{Introduction}

It is well known that generalized convexity notions are helpful for solving various (metrical) problems from non-Euclidean geometries in an elegant way. For example, Menger's notion of $d$-segments, yielding that of $d$-convex sets
(see Chapter II of \cite{Bo-Ma-So}), is a useful tool for solving location problems in finite dimensional real  Banach spaces (cf.\ \cite{Ma-Sw-We}). Another example, also referring to normed spaces, is the notion of ball convexity: usual convexity is extended by considering suitably defined intersections of balls instead of intersections of half-spaces. The ball hull of a given point set $S$ is the intersection of all balls (of fixed radius)
which contain $S$, and $S$ is called ball convex if it coincides with its ball hull. Ball convex sets are strongly related to notions from several recent research topics, such as ball polytopes, applications of spindle convexity, bodies of constant width, and diametrically maximal (or complete) sets. In the present article we study geometric properties and (minimal) representations of ball convex bodies in normed spaces. In terms of ball convexity and related notions, we derive separation properties of ball convex bodies, various characterizations of strictly convex norms, and an application for diametrically maximal sets, which answers a question from \cite{martini_et_al_2014}. Introducing suitable notions describing the boundary structure of ball convex bodies, our main results refer to minimal representations of ball convex bodies, particularly in terms of their ball exposed faces. More precisely, we extend the formula $K = {\rm cl}({\rm conv}
({\rm exp}(K)))$ from classical convexity (where $K$ is a convex body in $\bR^n$) to the concept of ball convexity in normed spaces. On the other hand, we derive theorems on the representation of ball convex bodies ``from inside''. That is, we show that unions of increasing sequences of ball convex bodies are, essentially, ball convex, and we present ball hulls of sets by unions of ball hulls of finite subsets. In that context we solve a problem from \cite{La-Na-Ta}.
We finish with some open questions inspired by the notions of ball hull and ball convexity; they refer to spindle convex sets and generalized Minkowski spaces (whose unit balls need not be centered at the origin).


We will give now a brief survey on what has been done regarding ball convexity and related notions. Intersections of finitely many congruent Euclidean balls were studied in \cite{Bie1} and \cite{Bie2}, in
three dimensions by \cite{Hepp, He-Re, Str, Grue}; see also \cite{Mar-Swa}.
The notions of ball hull and ball convexity have been considered by various authors,
defining them via intersections of balls of some fixed radius $R > 0$ and calling this concept also $R$-convexity; see, e.g., \cite{B-C-C, Bezd-Na, kupitz_et_al_2010, La-Na-Ta}. In view of this concept bodies of constant width (see \cite{Mon}), Minkowski sums, Hausdorff limits, and approximation properties of $R$-convex sets (cf.\ \cite{Pol1,Pol2,Po-Ba}) are investigated. Also analogues of the Krein--Milman theorem and of Carath\'eodory's theorem
(see \cite{Pol1,Pol3}) are considered, but only for the Euclidean norm. Not much has been done for normed spaces; however, for related results we refer to \cite{Bal} for Hilbert spaces and to \cite{Ba-Po,Ali,Ba-Iv 1,Mart-Spir} for normed spaces. A recent contribution is \cite{La-Na-Ta},
referring, e.g., to the Banach--Mazur distance and Hadwiger illumination numbers of sets being ball convex in the sense described here.

Closely related is the concept of ball polytopes. It was investigated in \cite{bezdek_et_al_2007,kupitz_et_al_2010,Pa} (but see also \cite{Pol1}, \cite[Chapter 6]{BezdekK}, and \cite[Chapter 5]{Bezdek}). The
boundary structure of ball polytopes
is interesting (digonal facets can occur, and hence their edge-graphs are different from usual polyhedral edge-graphs), their properties are also useful for constructing
bodies of constant width, and analogues of classical theorems like those  of Carath\'eodory and Steinitz on linear convex hulls are proved in these papers.

The study of the related notion of spindle convexity (also called hyperconvexity or $K$-convexity) was initiated by Mayer \cite{May}; see also \cite{meissner1911} and, for Minkowski spaces,
\cite[p. 99]{Val}. The definition is given in \S~8 below. For a discussion of this notion we refer to the survey
\cite[p.\ 160]{Da-Gr-Kl} and, for further results and references in the spirit of abstract convexity and combinatorial geometry, to the papers \cite{bezdek_et_al_2007,La-Na-Ta,Pa,Bezd, FV} and \cite[Chapters 5 and 6]{Bezdek}. In \cite{BN} this notion was extended to analogues of starshaped sets.

To avoid confusion, we shortly mention another concept which is also called ball convexity. Namely, in \cite{La 1} a set is called ball convex if, with any finite number of points, it contains the intersection of all balls (of arbitrary radii) containing the points. The ball convex hull of a set $S$ is again defined as the intersection of all ball convex sets containing $S$. In \cite{La 1,La 2} this notion was investigated for normed spaces, and in
\cite{La 3} the relations of these notions to metric or $d$-convexity are investigated.
The ball hull mapping studied for Banach spaces in \cite{Mo-Sch3,Mo-Sch4} is also related.


\section{Definitions and notations\label{sec_definitions}}

Let $\mathcal{K}^n= \{S \subseteq \bR^n: S$ is compact, convex, and non-empty$\}$ be the set of all \emph{convex bodies} in $\bR^n$ (thus, in our terminology, a convex body need not have interior points). Let $B \in \mathcal{K}^n$ be centered at the origin $o$ of $\bR^n$ and have non-empty interior. We denote by $(\bR^n, \| \cdot \|)$ the $n$-dimensional \emph{normed} or
\emph{Minkowski space} with unit ball $B$, i.e., the $n$-dimensional real Banach space whose \emph{norm} is given by $\| x \| = \min \{\lambda \ge 0 : x \in \lambda B\}$. Any homothetical copy $B(x,r), x \in \bR^n$ and
$r \ge 0$, of $B$ is a \emph{closed ball} of $(\bR^n, \| \cdot \|)$ with center $x$ and radius $r$; therefore we  replace $B$ by writing from now on $B(o,1)$ for
the \emph{unit ball of} $(\bR^n, \| \cdot \|$). The boundary of the ball $B(x,r)$ is the \emph{sphere} $S(x,r)$, and
therefore $S(o,1)$ denotes the
\emph{unit sphere} of our Minkowski space. Note that we will use the symbol $S$ for an arbitrarily given point set in $\bR^n$. For a compact $S$, we write dist$(x,S) = \min \{\|x-y\| : y \in S\}$ for the
\emph{distance of $x$ and} $S$, and we denote by rad$(S)$ the \emph{circumradius} of $S$, i.e., the radius of any \emph{circumball} (or minimal enclosing ball) of $S$, whose existence is assured by the boundedness of $S$. The
\emph{diameter} of $S$ is given by ${\rm diam}(S) = \max \{\|x-y\| : x,y \in S\}$. The triangle inequality yields the left-hand side of
\begin{equation}\label{(1a)}
\frac{1}{2} {\rm diam} (S) \le {\rm rad} (S) \le \frac{n}{n+1} {\rm diam} (S)\,,
\end{equation}
and we refer to \cite[Theorem 6]{bohnenblust} for the right-hand side.

As usual, we use the abbreviations int$(S)$, cl$(S)$, bd$(S)$, conv$(S)$, and aff$(S)$ for \emph{interior, closure, boundary, convex hull}, and \emph{affine hull} of $S$, respectively. We write $[x_1,x_2]$ for the
\emph{closed segment} with endpoints $x_1, x_2 \in \bR^n$, and $(a,b)$, $(a,b]$, $[a,b]$ are analogously used for the \emph{open}, \emph{half-open} or \emph{closed interval} with $a,b \in \bR$, respectively. We use $|\cdot|$ for the
\emph{cardinality} of a set.

A convex body is called \emph{strictly convex} if its boundary does not contain proper segments; analogously, $\|\cdot\|$ is called a \emph{strictly convex norm} if the respective unit ball is strictly convex.

Since we want to derive results for generalized convexity notions, the following definitions yield direct analogues of notions from classical convexity; see \cite{schneider1993}.
The first of them is an analogue of the (closed) convex hull.
Namely, the \emph{ball hull} of a set $S$ is defined by
$$
\BH(S) = \bigcap_{S \subseteq B(x,1)} B(x,1) \,.
$$
A formally clearer expression would be $\BH(S)=\bigcap_{x \in \bR^n:\,S \subseteq B(x,1)} B(x,1)$, but we assume that the above shorter notation, as well as similar ones in the sequel, will not cause confusion. (We underline once more that here and below we use balls of radius $1$.) A \emph{ball convex} (\emph{b-convex}) \emph{set} $S$ is characterized by $S = {\rm bh}_1(S)$ or, equivalently, by the property that $S$ is an intersection of closed balls of radius $1$ (then $S$ is necessarily closed and convex). A \emph{b-convex body} $K$ is a bounded non-empty b-convex set (the analogue of a convex body in classical convexity); $\emptyset$ and $\bR^n$ are
the only b-convex sets that are not b-convex bodies. (Note that $\mathbb{R}^n$ is $b$-convex, since we want to understand the intersection of an empty family of sets as $\mathbb{R}^n$.)
A \emph{supporting sphere} $S(x,1)$ of $K$ is characterized by $K \subseteq B(x,1)$ and $K \cap S(x,1) \not= \emptyset$; the corresponding
\emph{exposed b-face} (or \emph{b-support set}) is $K \cap S(x,1)$ (note that non-empty \emph{facets} from \cite[Definition 5.3]{kupitz_et_al_2010}
 are a special case).

If an exposed b-face is a singleton $\{x_0\}$, then $x_0$ is called a \emph{b-exposed point of } $K$, and b-exp$(K)$ denotes the set of all b-exposed points. We note that several such concepts, referring to the
analogous notions for ball polytopes, their boundary structure, separation properties with respect to spheres etc., can be found in the papers \cite{Bezd, kupitz_et_al_2010}, but are
defined there only for the subcase of the Euclidean norm.
Finally, a set $S$ is called \emph{b-bounded} if $\RAD(S)< 1$. This means that $S$ is inside a ball of radius $1$ and separated from its bounding sphere, which plays the role of a hyperplane in classical convexity.

We close this section by summarizing several basic facts about ball hulls and circumradii, and we give a lemma on intersections of compact sets with the boundaries of their circumballs.

\begin{lemma}
\label{lem_basics}
Let $(\mathbb{R}^n,\|\cdot\|)$ be a Minkowski space. The following are satisfied for all $S,T \subseteq \mathbb{R}^n$ and $x \in \mathbb{R}^n$:
\begin{enumerate}[label={(\alph*)}]
\item $S \subseteq \CL(S) \subseteq \CL(\CONV(S)) \subseteq \BH(S)= \BH(\CL(S))= \BH(\CONV(S))= \BH(\BH(S))$.\label{bh_bhbh}
\item If $S \subseteq T$, then $\BH(S) \subseteq \BH(T)$.\label{bh_inclusion}
\item $B(x,r)$ is a b-convex body for every $r \in [0,1]$.\label{ball_is_body}
\item If $\RAD(S) \le 1$, then $\RAD(\BH(S))=\RAD(S)$.
In particular, $\BH(S)$ is b-bounded if $S$ is b-bounded.\label{rad_bh}
\item If $S$ is closed and $S \subseteq \INT(B(x,r))$ for some $r > 0$, then $S \subseteq B(x,r')$ for some $r' \in (0,r)$ and $\RAD(S) < r$. In particular, a closed subset of $\mathbb{R}^n$ is b-bounded if and only if it is covered by an open ball of radius $1$.\label{rad_interior}
\end{enumerate}
\end{lemma}

\begin{proof}
Parts \ref{bh_bhbh} and \ref{bh_inclusion} are obvious; see \cite[Lemma~1]{martini_et_al_2013} for a collection of related statements.

For \ref{ball_is_body}, the triangle inequality gives the following representation of $B(x,r)$ as an intersection of balls of radius $1$: $B(x,r)= \bigcap_{\|y-x\| \le 1-r} B(y,1)$.

To see \ref{rad_bh}, first note that $\RAD(S) \le \RAD(\BH(S))$ by \ref{bh_bhbh}. If $B(x,\RAD(S))$ is a circumball of $S$, then $\BH(S) \subseteq \BH(B(x,\RAD(S))) = B(x,\RAD(S))$ by \ref{bh_inclusion} and \ref{ball_is_body}. Hence
$\RAD(\BH(S)) \le \RAD(B(x,\RAD(S)))=\RAD(S)$.

For \ref{rad_interior}, suppose that $S$ contains at least two points. Consider the continuous function $f: S \rightarrow \mathbb{R}$, $f(y)=\DIST(y,S(x,r))=\DIST(y,\mathbb{R}^n \setminus B(x,r))$. Since $S$ is compact, $f$ attains its minimum: $f(y) \ge f(y_0) \in (0,r)$ for all $y \in S$. This shows that $\DIST(y,\mathbb{R}^n \setminus B(x,r)) \ge f(y_0)$ for all $y \in S$; i.e., $S \subseteq B(x,r')$, where $r'=r-f(y_0) \in (0,r)$.
\end{proof}

\begin{lemma}
\label{lem_circumintersection}
Let $B(x_0,\RAD(S))$ be a circumball of a non-empty compact subset $S$ of a Minkowski space $(\mathbb{R}^n,\|\cdot\|)$. Then $\RAD(S\cap S(x_0,\RAD(S)))= \RAD(S)$. In particular, there exist $x,x' \in
S\cap S(x_0,\RAD(S))$ such that $\|x-x'\| \ge \frac{n+1}{n} \RAD(S)$.
\end{lemma}

\begin{proof}
Without loss of generality, we set $B(x_0,\RAD(S))=B(o,1)$. Assume that, contrary to our claim, $\RAD(S \cap S(o,1))<1$. Then there exists $x_1 \in \mathbb{R}^n$ such that
\begin{equation}
\label{eq_ci1}
S \cap S(o,1) \subseteq \INT(B(x_1,1)).
\end{equation}
Since $\RAD(S)=1$, Lemma~\ref{lem_basics}\ref{rad_interior} gives points
\begin{equation}
\label{eq_ci2}
y_i \in S \setminus {\rm int}\left(B\left(\frac{1}{i}x_1,1\right)\right), \quad i =1,2,\ldots
\end{equation}
Note that
$$
(y_i)_{i=1}^\infty \subseteq S \setminus \INT(B(x_1,1)),
$$
because $y_i \in S \setminus \INT\left(B\left(\frac{1}{i}x_1,1\right)\right) \subseteq B(o,1) \setminus \INT\left(B\left(\frac{1}{i}x_1,1\right)\right)$ gives $\|y_i\| \le 1$, $\left\|y_i-\frac{1}{i}x_1\right\| \ge 1$,
and in turn
\begin{eqnarray*}
\|y_i-x_1\| &=&\left\|i\left(y_i-\frac{1}{i}x_1\right)-(i-1)y_i\right\|\\
&\ge& i\left\|y_i-\frac{1}{i}x_1\right\|-(i-1)\|y_i\|\\
&\ge& i-(i-1)\\
&=&1.
\end{eqnarray*}
Since $S \setminus \INT(B(x_1,1))$ is compact, $\left( y_i \right)_{i=1}^\infty$ has an accumulation point $y_0 \in S \setminus \INT(B(x_1,1))$. We know that $\|y_0\| \le 1$ from $S \subseteq B(o,1)$, whereas \eqref{eq_ci2} gives $\left\|y_i-\frac{1}{i}x_1\right\| \ge 1$ and, by $i \to \infty$, $\|y_0\| \ge 1$. This way we see that
$$
y_0 \in S \cap S(o,1) \setminus \INT(B(x_1,1)),
$$
which contradicts \eqref{eq_ci1} and completes the proof of $\RAD(S\cap S(x_0,\RAD(S)))=\RAD(S)$.

Now the additionally claimed existence of $x,x' \in S \cap S(x_0,\RAD(S))$ such that $\|x-x'\| \ge \frac{n+1}{n}\RAD(S)$ is a consequence of the right-hand estimate in (\ref{(1a)})
 and the compactness of $S$.
\end{proof}


\section{Separation properties}

The following results on the separation of b-convex bodies and points by spheres are analogues of theorems on the separation by hyperplanes in classical convexity.

\begin{proposition}
\label{prop_separation}
Let $K$ be a b-convex body in a Minkowski space $(\mathbb{R}^n,\|\cdot\|)$.
\begin{enumerate}[label=(\alph*)]
\item For every $x_0 \in \BD(K)$, there exists a supporting sphere
$S(y_0,1)$ of $K$ such that $x_0 \in S(y_0,1)$.\label{sep_a}
\item For every $x_0 \in \mathbb{R}^n \setminus K$, there exists a supporting sphere $S(y_0,1)$ of $K$ such that $x_0 \notin B(y_0,1)$.\label{sep_b}
\item If $K$ is b-bounded then, for every $x_0 \in \mathbb{R}^n \setminus K$, there exists a sphere of unit radius $S(y_0,1)$ such that $K \subseteq \INT(B(y_0,1))$ and $x_0 \notin B(y_0,1)$. In particular, $K \subseteq B(y_0,r)$ for some $r \in (0,1)$.\label{sep_c}
\end{enumerate}
\end{proposition}

\begin{proof}
For proving \ref{sep_a}, note that the assumption
$$
x_0 \in \BD(K)= \BD(\BH(K))= \BD\left(\bigcap_{K \subseteq B(y,1)} B(y,1)\right)
$$
yields the existence of a sequence $(y_i)_{i=1}^\infty \subseteq \mathbb{R}^n$ such that $K \subseteq B(y_i,1)$ for all $i$ and
$$
0= \lim_{i \to \infty} \DIST(x_0, \mathbb{R}^n \setminus B(y_i,1))= \lim_{i \to \infty} (1-\|x_0-y_i\|).
$$
By compactness, $(y_i)_{i=1}^\infty$ has a convergent subsequence, and we can assume that $\lim_{i \to \infty} y_i= y_0$ without loss of generality. Then the above observations imply $K \subseteq B(y_0,1)$ and $\|x_0-y_0\|=1$, i.e., $x_0 \in S(y_0,1)$. This is our claim.

For \ref{sep_b}, we have $x_0 \notin K= \bigcap_{K \subseteq B(y,1)} B(y,1)$. Hence there is $y_1 \in \mathbb{R}^n$ such that $K \subseteq B(y_1,1)$ and $x_0 \notin B(y_1,1)$. We consider the translated balls $B_\lambda:=B(y_1+\lambda(y_1-x_0),1)$, $\lambda \ge 0$. We know that $K \subseteq B_0$. Let $\lambda_0 \ge 0$ be maximal such that
$$
K \subseteq B_\lambda \quad\mbox{ for }\quad 0 \le \lambda \le \lambda_0.
$$
By the maximality of $\lambda_0$, $\BD(B_{\lambda_0})=S(y_1+\lambda_0(y_1-x_0),1)=:S(y_0,1)$ is a supporting sphere of $K$. Moreover, $x_0 \notin B(y_0,1)$, because
$x_0 \notin B(y_1,1)$ gives
$$
\|x_0-y_0\|=\|x_0-(y_1+\lambda_0(y_1-x_0))\|=(1+\lambda_0)\|x_0-y_1\|> 1+\lambda_0\ge 1.
$$
This proves \ref{sep_b}.

For the proof of \ref{sep_c}, the b-boundedness of $K$ gives $y_1 \in \mathbb{R}^n$ such that $K \subseteq \INT(B(y_1,1))$. By \ref{sep_b}, there is $y_2 \in \mathbb{R}^n$ with $K \subseteq B(y_2,1)$ and $x_0 \notin B(y_2,1)$. We can pick $\varepsilon \in (0,1)$ small enough such that
$$
x_0 \notin B(y_0,1), \quad\mbox{ where }\quad y_0:=y_2+\varepsilon(y_1-y_2).
$$
Then we obtain
\begin{equation}
\label{eq_K_int}
K \subseteq \INT(B(y_0,1)),
\end{equation}
because, for arbitrary $x \in K$, the inclusions $K \subseteq \INT(B(y_1,1))$ and $K \subseteq B(y_2,1)$ imply $\|x-y_1\|<1$, $\|x-y_2\| \le 1$, and in turn
\begin{eqnarray*}
\|x-y_0\|
&=& \|x-(y_2+\varepsilon(y_1-y_2))\| \\
&=& \|\varepsilon(x-y_1)+(1-\varepsilon)(x-y_2)\| \\
&\le& \varepsilon \|x-y_1\|+(1-\varepsilon)\|x-y_2\| \\
&<& \varepsilon+(1-\varepsilon) \\
&=& 1.
\end{eqnarray*}
Finally, \eqref{eq_K_int} yields $K \subseteq B(y_0,r)$ for suitable $r \in (0,1)$ by Lemma~\ref{lem_basics}\ref{rad_interior}.
\end{proof}

\begin{corollary}
Every b-convex body in a Minkowski space $(\mathbb{R}^n,\|\cdot\|)$ satisfies
$$
\BD(K)= \bigcup \{F: F \mbox{ is an exposed b-face of } K\}.
$$
\end{corollary}

\begin{proof}
Proposition~\ref{prop_separation}\ref{sep_a} gives ``$\subseteq$''. The converse inclusion is implied by the definition of exposed b-faces.
\end{proof}

Proposition~\ref{prop_separation} gives rise to alternative representations of ball hulls.

\begin{corollary}
\label{cor_hull_alternative1}
Every b-bounded subset $S$ of a Minkowski space $(\mathbb{R}^n,\|\cdot\|)$ satisfies
$$
\BH(S)=
\bigcap_{S \subseteq \INT(B(x,1))} B(x,1)=
\bigcap_{S \subseteq B(x,r), r<1} B(x,1)=
\bigcap_{S \subseteq B(x,r), r < 1} B(x,r).
$$
\end{corollary}

\begin{proof}
We assume that $S \ne \emptyset$ and put $A:=\BH(S)=\bigcap_{S \subseteq B(x,1)} B(x,1)$, $B:=\bigcap_{S \subseteq \INT(B(x,1))} B(x,1)$, $C:=\bigcap_{S \subseteq B(x,r), r<1} B(x,1)$ and $D:=\bigcap_{S \subseteq B(x,r), r < 1} B(x,r)$. The inclusions $A \subseteq B \subseteq C$ and $D \subseteq C$ are trivial. It suffices to prove that $C \subseteq A$ and $A \subseteq D$.

For proving $C \subseteq A$, we consider an arbitrary $x_0 \in \mathbb{R}^n \setminus A$ and have to show that $x_0 \notin C$. Application of Proposition~\ref{prop_separation}\ref{sep_c} to $A$,
which is b-bounded by Lemma~\ref{lem_basics}\ref{rad_bh},
and $x_0$ gives $y_0 \in \mathbb{R}^n$ and $r \in (0,1)$ such that
$S \subseteq A \subseteq B(y_0,r)$ and $x_0 \notin B(y_0,1)$. This yields $x_0 \notin C$.

For $A \subseteq D$, note that
$$
S \subseteq B(x,r) \quad\Leftrightarrow\quad A \subseteq B(x,r).
$$
Indeed, if $S \subseteq B(x,r)$, then, by Lemma~\ref{lem_basics}\ref{bh_inclusion} and \ref{ball_is_body},
$A=\BH(S)\subseteq \BH(B(x,r))=B(x,r)$. Conversely, if $A \subseteq B(x,r)$, then $S \subseteq \BH(S)= A \subseteq B(x,r)$ by Lemma~\ref{lem_basics}\ref{bh_bhbh}.

The above equivalence yields
$$
A \subseteq \bigcap_{A \subseteq B(x,r),r < 1} B(x,r) = \bigcap_{S \subseteq B(x,r),r < 1} B(x,r) = D.
$$
\end{proof}

\begin{corollary}
\label{cor_hull_alternative2}
Every b-bounded closed subset $S$ of a Minkowski space $(\mathbb{R}^n,\|\cdot\|)$ satisfies
$$
\BH(S)=
\bigcap_{S \subseteq \INT(B(x,1))} \INT(B(x,1)).
$$
\end{corollary}

\begin{proof}
By Corollary~\ref{cor_hull_alternative1},
$$
\BH(S)=\bigcap_{S \subseteq \INT(B(x,1))} B(x,1) \supseteq \bigcap_{S \subseteq \INT(B(x,1))} \INT(B(x,1)).
$$

For the converse inclusion, note that
$$
S \subseteq \INT(B(x,1)) \quad\Rightarrow\quad
\BH(S) \subseteq \INT(B(x,1)).
$$
Indeed, if $S \subseteq \INT(B(x,1))$, then $S \subseteq B(x,r)$ for some $r \in (0,1)$ by Lemma~\ref{lem_basics}\ref{rad_interior}, and, by Lemma~\ref{lem_basics}\ref{bh_inclusion} and \ref{ball_is_body}, $\BH(S)
\subseteq \BH(B(x,r)) = B(x,r) \subseteq \INT(B(x,1)).$

The above implication yields
$$
\bigcap_{S \subseteq \INT(B(x,1))} \INT(B(x,1)) \supseteq \bigcap_{\BH(S) \subseteq \INT(B(x,1))} \INT(B(x,1)) \supseteq \BH(S).
$$
\end{proof}

The assumption of b-boundedness is essential in Corollaries~\ref{cor_hull_alternative1} and \ref{cor_hull_alternative2}. For example, if $S$ is a closed ball of radius $1$, then $\BH(S)=S$, whereas the four other intersections represent $\mathbb{R}^n$, since they are intersections over empty index sets.

To see that the assumption of closedness in Corollary~\ref{cor_hull_alternative2} cannot be dropped, consider the example $S=\INT(B(x_0,r_0))$ with $x_0 \in \mathbb{R}^n$ and $r_0 \in (0,1)$. Then
$$
\BH(\INT(B(x_0,r_0)))=\BH(B(x_0,r_0))=B(x_0,r_0)
$$
by Lemma~\ref{lem_basics}\ref{bh_bhbh} and \ref{ball_is_body}. In contrast to that,
$$
\bigcap_{\INT(B(x_0,r_0)) \subseteq \INT(B(x,1))} \INT(B(x,1))=\bigcap_{\|x-x_0\| \le 1-r_0} \INT(B(x,1))=\INT(B(x_0,r_0)),
$$
as can be checked by the triangle inequality.

In classical convexity two disjoint convex sets can be separated by a hyperplane. The analogous claim for ball convexity would say that, given two disjoint b-convex bodies $K_1, K_2 \subseteq \bR^n$, there exists a \emph{separating sphere} $S(x_0,1)$ for $K_1$ and $K_2$; i.e., $K_1 \subseteq B(x_0,1)$ and $K_2 \cap B(x_0,1)= \emptyset$. In fact, one knows even more if the underlying Minkowski space $(\bR^n,\|\cdot\|)$ is a Euclidean space (see \cite[Lemma~3.1 and Corollary 3.4]{bezdek_et_al_2007}), if its unit ball is a cube (see \cite[Corollary~3.15]{La-Na-Ta}), or if it is two-dimensional (see \cite[Theorem~4]{La-Na-Ta}). Then, for every b-convex body $K$ and every supporting hyperplane $H$ of $K$, there exists a sphere $S(x_0,1)$ such that $K \subseteq B(x_0,1)$ and $\INT(B(x_0,1)) \cap H= \emptyset$. However, the last statement fails in general (see \cite[Example~3.9]{La-Na-Ta} for an example in a generalized Minkowski space whose unit ball is not centrally symmetric). Here we show that even the (formally weaker) separation of two b-convex bodies by a unit sphere may fail in a (symmetric) Minkowski space.

\begin{example}\label{ex-l1}
Let $l_1^3$ be the three-dimensional Minkowski space with unit ball $B(o,1)=\CONV(\{(\pm 1,0,0), (0,\pm 1,0), (0,0,\pm 1)\})$, let $0 < \varepsilon < \frac{1}{2}$, and consider the segments $K_1= \left[\left(\frac{1}{4},\frac{1}{4},0\right),\left(-\frac{1}{4},-\frac{1}{4},0\right)\right]$ and $K_2= \left[\left(\frac{1}{4},-\frac{1}{4},\varepsilon\right),\left(-\frac{1}{4},\frac{1}{4},\varepsilon\right)\right]$. Then $K_1$ and $K_2$ are disjoint b-bounded b-convex bodies in $l_1^3$, and there is no unit sphere $S(x_0,1)$ such that $K_1 \subseteq B(x_0,1)$ and $K_2 \cap \INT(B(x_0,1))= \emptyset$.
\end{example}

\begin{proof}
$K_1$ is b-convex, because $K_1=B\left(\left(-\frac{3}{4},\frac{1}{4},0\right),1\right) \cap B\left(\left(\frac{3}{4},-\frac{1}{4},0\right),1\right)$, and b-bounded, since  $\RAD(K_1)=\frac{1}{2}$. Similarly, $K_2$ is b-bounded and b-convex.

If $K_1 \subseteq B(x_0,1)$, then $B(x_0,1)$ contains at least one of the points of the segment $\left[\left(\frac{1}{4},-\frac{1}{4},\frac{1}{2}\right),\left(-\frac{1}{4},\frac{1}{4},\frac{1}{2}\right)\right]$, and we obtain $K_2 \cap \INT(B(x_0,1)) \ne \emptyset$, since $0 < \varepsilon < \frac{1}{2}$.
\end{proof}


\section{Characterizations of strict convexity}

Some of our results will require strict convexity of the norm $\|\cdot\|$. On the other hand, strict convexity can be reflected by numerous properties related to concepts introduced in Section~\ref{sec_definitions}.
This fourth section here is devoted to characterizations of strict convexity. We start with characterizations by properties of balls, circumballs and circumradii; for (iv) and (v) in the following lemma we also refer to \cite{AZ, Ma-Swa-Weiss}.

\begin{lemma}
\label{lem_str_conv}
Let $(\mathbb{R}^n,\|\cdot\|)$ be a Minkowski space. The following are equivalent:
\begin{enumerate}[label={(\roman*)}]
\item $\|\cdot\|$ is strictly convex. \label{l_i}
\item Each supporting hyperplane of a closed ball meets that ball in exactly one point. \label{l_ii}
\item The circumradius of the intersection of any two distinct balls of the same radius $r > 0$ is smaller than $r$. \label{l_iii}
\item Every bounded non-empty subset of $\mathbb{R}^n$ has a unique circumball. \label{l_iv}
\item For any two distinct points $x_1,x_2 \in \mathbb{R}^n$, $\{x_1,x_2\}$ has a unique circumball. \label{l_v}
\end{enumerate}
\end{lemma}

\begin{proof}
\ref{l_i}$\Rightarrow$\ref{l_ii}: If \ref{l_ii} fails, then some ball meets one of its supporting hyperplanes in at least two distinct points $x_1,x_2$. Then the segment $[x_1,x_2]$ is contained in the boundary of that ball, contradicting \ref{l_i}.

\ref{l_ii}$\Rightarrow$\ref{l_i}: If \ref{l_i} fails, then the boundary of $B(o,1)$ contains a line segment $L$ of positive length. The disjoint convex sets $\INT(B(o,1))$ and $L$ can be separated by a hyperplane $H$. Then $H$ is a supporting hyperplane of $B(o,1)$ and contains $L$, contradicting \ref{l_ii}.

\ref{l_i}$\Rightarrow$\ref{l_iii}: If $x_1, x_2 \in \mathbb{R}^n$ are distinct points and if $r > 0$, then
\begin{equation}
\label{eq_intersection_inclusion}
B(x_1,r) \cap B(x_2,r) \subseteq {\rm int}\left(B\left(\frac{x_1+x_2}{2},r\right)\right),
\end{equation}
which implies our claim $\RAD(B(x_1,r) \cap B(x_2,r))<r$ by Lemma~\ref{lem_basics}\ref{rad_interior}. To verify \eqref{eq_intersection_inclusion}, assume the contrary; i.e.,
$\left\|x-\frac{x_1+x_2}{2}\right\| \ge r$ for some $x \in B(x_1,r) \cap B(x_2,r)$. Then
$$
r \le \left\|x-\frac{x_1+x_2}{2}\right\| \le \frac{1}{2} (\|x-x_1\|+\|x-x_2\|) \le \frac{1}{2}(r+r)=r,
$$
hence all terms in the above estimate agree and we obtain $\|x-x_1\|=\|x-x_2\|=\left\|x-\frac{x_1+x_2}{2}\right\|=r$. This shows that $[x_1,x_2]$ is a segment in $S(x,r)$, contradicting \ref{l_i} and proving \eqref{eq_intersection_inclusion}.

\ref{l_iii}$\Rightarrow$\ref{l_iv}: If \ref{l_iv} fails, then there is a bounded set $S$ with circumradius $\RAD(S)>0$ that has two circumballs $B(x_1,\RAD(S))$ and $B(x_2,\RAD(S))$, $x_1 \ne x_2$. This implies $B(x_1,\RAD(S)) \cap B(x_2,\RAD(S)) \supseteq S$ and $\RAD(B(x_1,\RAD(S)) \cap B(x_2,\RAD(S))) \ge \RAD(S)$, contradicting \ref{l_iii}.

For \ref{l_iv}$\Leftrightarrow$\ref{l_v}$\Leftrightarrow$\ref{l_i}, see \cite[Lemma~1.2]{AZ}.
\end{proof}

Now we come to characterizations of strict convexity of norms in terms of concepts related to b-convexity that are defined in Section~\ref{sec_definitions}.

\begin{proposition}
\label{prop_str_conv}
Let $(\mathbb{R}^n,\|\cdot\|)$ be a Minkowski space. The following are equivalent:
\begin{enumerate}[label={(\roman*)}]
\item The norm $\|\cdot\|$ is strictly convex. \label{p_i} \setcounter{enumi}{5}
\item Every b-convex body that is not b-bounded is a closed ball of radius $1$. \label{p_vi}
\item Every b-convex body that is not b-bounded has only one supporting sphere. \label{p_vii}
\item For every boundary point $x$ of a b-convex body $K$ that is not b-bounded, there exists only one supporting sphere of $K$ that contains $x$. \label{p_viii}
\item For every $x \in \mathbb{R}^n$, every $r \in (0,1)$ and every $x_0 \in \BD(B(x,r))$, $B(x,r)$ has only one supporting sphere that contains $x_0$. \label{p_ix}
\item There exist $x \in \mathbb{R}^n$ and $r \in (0,1)$ such that, for every $x_0 \in \BD(B(x,r))$, $B(x,r)$ has only one supporting sphere that contains $x_0$. \label{p_x}
\item For every $x \in \mathbb{R}^n$ and every $r \in (0,1)$, each supporting sphere of $B(x,r)$ meets $B(x,r)$ in only one point. \label{p_xi}
\item There exist $x \in \mathbb{R}^n$ and $r \in (0,1)$ such that each supporting sphere of  $B(x,r)$ meets $B(x,r)$ in only one point. \label{p_xii}
\item For every $x \in \mathbb{R}^n$ and every $r \in (0,1)$,
$\BEXP(B(x,r))=S(x,r)$. \label{p_xiii}
\item There exist $x \in \mathbb{R}^n$ and $r \in (0,1)$ such that $\BEXP(B(x,r))=S(x,r)$. \label{p_xiv}
\item Every b-convex body is strictly convex. \label{p_xv}
\item For any two distinct points $x_1,x_2 \in \mathbb{R}^n$,
$\BH(\{x_1,x_2\})$ is strictly convex. \label{p_xvi}
\item Every b-convex body that contains at least two points has non-empty interior. \label{p_xvii}
\item For any two distinct points $x_1,x_2 \in \mathbb{R}^n$,
$\INT(\BH(\{x_1,x_2\}))$ is non-empty. \label{p_xviii}
\end{enumerate}
\end{proposition}

\begin{proof}
The implications \ref{p_vi}$\Rightarrow$\ref{p_vii}$\Rightarrow$\ref{p_viii}, \ref{p_ix}$\Rightarrow$\ref{p_x}, \ref{p_xi}$\Rightarrow$\ref{p_xii},
\ref{p_xiii}$\Rightarrow$\ref{p_xiv},\\ \ref{p_xv}$\Rightarrow$\ref{p_xvi}, and \ref{p_xvii}$\Rightarrow$\ref{p_xviii} are obvious.

\ref{p_i}$\Rightarrow$\ref{p_vi}: Every b-convex body $K$ is a non-empty intersection of a non-empty family of closed balls of radius $1$. If the family consisted of more than one ball, then its intersection $K$ would be b-bounded by Lemma~\ref{lem_str_conv}\ref{l_i}$\Rightarrow$\ref{l_iii}. Hence the only b-convex bodies that are not b-bounded are closed balls of radius $1$.

\ref{p_viii}$\Rightarrow$\ref{p_i}: Suppose that \ref{p_i} fails.
Then condition \ref{l_iii} from Lemma~\ref{lem_str_conv} fails as well, and there are two points $x_1 \ne x_2$ such that $\RAD(B(x_1,1) \cap B(x_2,1))=1$. The body $K=B(x_1,1) \cap B(x_2,1)$ shows that \ref{p_viii} fails, too, because every $x \in \BD(B(x_1,1)) \cap \BD(B(x_2,1))$ belongs to $\BD(K)$ and has the two supporting spheres $S(x_1,1)$ and $S(x_2,1)$.

\ref{p_i}$\Rightarrow$\ref{p_ix} and \ref{p_i}$\Rightarrow$\ref{p_xi}: We use the fact that \emph{if two balls $B(y,s)$ and $B(y',s')$ of positive radii
in a strictly convex Minkowski space are on the same side of a common supporting hyperplane $H$ with respective touching points $y_0$ and $y'_0$, then the dilatation $\varphi$ that is uniquely determined by $\varphi(y_0)=y'_0$ and the dilatation factor $\frac{s'}{s}$ maps $B(y,s)$ onto $B(y',s')$.} To see this, consider the homotheties $\delta$ and $\delta'$ that map $B(y,s)$ and $B(y',s')$ onto $B(o,1)$, respectively. Then $\delta(H)=\delta'(H)$, because $B(o,1)$ has only one supporting hyperplane with the same outer normal vector as $H$, and $\delta(y_0)=\delta'(y'_0)$, since $\delta(H)=\delta'(H)$ has only one touching point with $B(o,1)$ (see Lemma~\ref{lem_str_conv}).
Now $\varphi=(\delta')^{-1} \circ \delta$, and the fact is verified.

Coming back to the proof of \ref{p_i}$\Rightarrow$\ref{p_ix} and \ref{p_i}$\Rightarrow$\ref{p_xi}, we consider an arbitrary supporting sphere $S(y,1)$ of $B(x,r)$ and suppose that $x_0$ belongs to the b-support set $B(x,r) \cap S(y,1)$. The supporting hyperplane of $B(y,1)$ at $x_0$ supports $B(x,r)$ as well. Now the above fact says that $B(y,1)$ is the image of $B(x,r)$ under the dilatation $\varphi$ with fixed point $x_0$ and factor $\frac{1}{r}$. This shows in particular that the supporting sphere $S(y,1)$ is uniquely determined by the touching point $x_0$, which proves \ref{p_ix} (because there exists at least one supporting sphere at $x_0$ according to Proposition~\ref{prop_separation}\ref{sep_a}). To show \ref{p_xi}, we have to prove that every point $x_1 \in B(x,r) \cap S(y,1)$ coincides with $x_0$. By the same argument as above, $B(y,1)$ is the image of $B(x,r)$ under the dilatation $\psi$ with fixed point $x_1$ and factor $\frac{1}{r}$. We obtain
$$
y=\varphi(x)=x_0+\frac{1}{r}(x-x_0) \quad\mbox{ and }\quad
y=\psi(x)=x_1+\frac{1}{r}(x-x_1),
$$
which gives $x_0=x_1$ and completes the proof of \ref{p_xi}.

\ref{p_x}$\Rightarrow$\ref{p_i}: Suppose that \ref{p_i} fails. Then, for every $x \in \mathbb{R}^n$ and every $r \in (0,1)$, $S(x,r)$ contains a line segment $[x_0,x_1] \subseteq S(x,r)$, $x_0 \ne x_1$. If $\varphi_0$ and $\varphi_1$ are dilatations with factor $\frac{1}{r}$ and fixed points $x_0$ and $x_1$, respectively, then $\varphi_0(S(x,r))$ and $\varphi_1(S(x,r))$ are distinct supporting spheres of $B(x,r)$. Clearly, we have
$$
x_0= \varphi_0(x_0) \in \varphi_0([x_0,x_1]) \subseteq \varphi_0(S(x,r)).
$$
Moreover,
$$
x_0= \varphi_1(rx_0+(1-r)x_1) \in \varphi_1([x_0,x_1]) \subseteq \varphi_1(S(x,r)).
$$
Hence $\varphi_0(S(x,r))$ and $\varphi_1(S(x,r))$ both are supporting spheres of $B(x,r)$ at $x_0 \in \BD(B(x,r))$, and \ref{p_x} is disproved.

\ref{p_xi}$\Rightarrow$\ref{p_xiii} and \ref{p_xii}$\Rightarrow$\ref{p_xiv} follow from Proposition~\ref{prop_separation}\ref{sep_a}.

\ref{p_xiv}$\Rightarrow$\ref{p_i}: If \ref{p_i} fails, then every ball $B(x,r)$, $x \in \mathbb{R}^n$, $r \in (0,1)$, contains a segment $[x_1,x_2]$, $x_1 \ne x_2$, in its boundary $S(x,r)$. We shall show that $\frac{x_1+x_2}{2} \notin \BEXP(B(x,r))$, this way disproving \ref{p_xiv}. Indeed, if $S(y,1)$ is a supporting sphere of $B(x,r)$ with touching point $\frac{x_1+x_2}{2} \in S(y,1)$, then $[x_1,x_2] \subseteq B(x,r) \subseteq B(y,1)$,
and the point $\frac{x_1+x_2}{2}$ (from the relative interior) of $[x_1,x_2]$ is in $S(y,1)=\BD(B(y,1))$. Hence $[x_1,x_2] \subseteq S(y,1)$. This shows that the exposed b-face $B(x,r) \cap S(y,1)$ that contains $\frac{x_1+x_2}{2}$ must necessarily cover the whole segment $[x_1,x_2]$. Thus $\frac{x_1+x_2}{2} \notin \BEXP(B(x,r))$.

\ref{p_i}$\Rightarrow$\ref{p_xv}: Assume that \ref{p_xv} fails; i.e., there are a b-convex body $K$ and two points $x_1 \ne x_2$ such that $[x_1,x_2] \subseteq \BD(K)$. By Proposition~\ref{prop_separation}\ref{sep_a}, there is a supporting sphere $S(y,1)$ of $K$ such that $\frac{x_1+x_2}{2} \in S(y,1)$. As above, we have $[x_1,x_2] \subseteq K \subseteq B(y,1)$ and $\frac{x_1+x_2}{2} \in S(y,1)=\BD(B(y,1))$, which yields $[x_1,x_2] \subseteq S(y,1)$ and contradicts \ref{p_i}.

\ref{p_xv}$\Rightarrow$\ref{p_xvii} and \ref{p_xvi}$\Rightarrow$\ref{p_xviii}: If a convex body $K$ contains two distinct points $x_1$ and $x_2$ and has empty interior, then $K$ is not strictly convex, because $[x_1,x_2] \subseteq K = \BD(K)$. This yields the two implications.

\ref{p_xviii}$\Rightarrow$\ref{p_i}: If \ref{p_i} fails, then there are $x_1 \ne x_2$ such that $[x_1,x_2] \subseteq S(o,1)$. Hence
$$
\begin{array}{rcl}
\left[ \frac{x_1-x_2}{2}, \frac{x_2-x_1}{2} \right]
&=& \left([x_1,x_2] -\frac{x_1+x_2}{2} \right) \cap \left( -[x_2,x_1]+\frac{x_1+x_2}{2} \right)\\[1ex]
&\subseteq& \left(S(o,1) -\frac{x_1+x_2}{2} \right) \cap \left( S(o,1)+\frac{x_1+x_2}{2} \right)\\[1ex]
&=& S\left(-\frac{x_1+x_2}{2},1 \right) \cap  S\left(\frac{x_1+x_2}{2},1 \right)\\[1ex]
&=& B\left(-\frac{x_1+x_2}{2},1 \right) \cap  B\left(\frac{x_1+x_2}{2},1 \right).
\end{array}
$$
The last equation is a consequence of $\left\|\frac{x_1+x_2}{2}-\left(-\frac{x_1+x_2}{2}\right)\right\|=2\left\|\frac{x_1+x_2}{2}\right\|=2$. We obtain
$$
\begin{array}{rcl}
\INT\!\left(\BH\left(\left\{\frac{x_1-x_2}{2}, \frac{x_2-x_1}{2}\right\}\right)\right)
&\subseteq&
\INT\!\left(\BH\!\left(B\left(-\frac{x_1+x_2}{2},1 \right) \cap B\left(\frac{x_1+x_2}{2},1 \right)\right)\right)\\[1ex]
&=& \INT\!\left(B\left(-\frac{x_1+x_2}{2},1 \right) \cap B\left(\frac{x_1+x_2}{2},1 \right)\right)\\[1ex]
&=& \INT\!\left(S\left(-\frac{x_1+x_2}{2},1 \right) \cap S\left(\frac{x_1+x_2}{2},1 \right)\right)\\[1ex]
&=& \emptyset,
\end{array}
$$
which contradicts \ref{p_xviii}.
\end{proof}


\section{Representation of ball hulls ``from inside''}

In this section we deal with b-convexity of unions of increasing sequences of b-convex bodies and with the representation of ball hulls of sets by unions of ball hulls of finite subsets. We start with an auxiliary statement.

\begin{lemma}\label{lemma1}
Let $K \in \cK^n$, let $H \subseteq \bR^n$ be a hyperplane, and let $(y_i)^\infty_{i = 1} \subseteq \bR^n$ such that
$y_i \!\stackrel{{\scriptstyle i \to \infty}}{\longrightarrow} \! y_0 \in \bR^n$ and $H \cap (K + y_i) \not= \emptyset$ for all $i = 1,2, \dots$ Then $H \cap (K + y_i)
\! \stackrel{{\scriptstyle i \to \infty}}{\longrightarrow} \! H \cap (K + y_0)$ in the
Hausdorff distance.
\end{lemma}

\begin{proof}
First note that $H \cap (K+y_0) \not= \emptyset$, and in turn $H \cap (K+y_0) \in \cK^n$. Indeed,
for every $i \ge 1$, we can pick $z_i \in H \cap (K+y_i)$, i.e., $z_i = x_i + y_i$ with $x_i \in K$. By the compactness of $K$ we see that,
w.l.o.g., $x_i \, \stackrel{{\scriptstyle i \to \infty}}{\longrightarrow} \, x_0 \in K$. We get
$z_0 : = x_0 + y_0 = \lim_{i \to \infty} z_i \in H \cap (K+y_0)$, because $H$ is closed.

By \cite[Theorem 1.8.7]{schneider1993}, our claim $H \cap (K+y_i) \! \stackrel{{\scriptstyle i \to \infty}}{\longrightarrow} \! H \cap (K + y_0)$ is now equivalent to the following conditions taken together:
\begin{enumerate}
\item[(a)] for every $t_0 \in H \cap (K+y_0)$, there exist $t_i \in H \cap (K+y_i)$, $i = 1,2,\dots$, such that $t_i \! \stackrel{{\scriptstyle i \to \infty}}{\longrightarrow} \! t_0$;
\item[(b)] if $(t_{i_j})^\infty_{j=1}$ is a sequence with $i_1 < i_2 < \dots$, $t_{i_j} \in H \cap (K+y_{i_j})$, and $t_{i_j} \! \stackrel{{\scriptstyle j \to \infty}}{\longrightarrow} \! t_0 \in \bR^n$,
then $t_0 \in H  \cap (K+y_0)$.
\end{enumerate}

\emph{Proof of (a).} Suppose that $H = \{x \in \bR^n: \langle u,x \rangle = c\}$, $\langle \cdot,\cdot \rangle$ denoting the usual scalar product. Then
fix $t_0 \in H \cap (K+y_0)$, i.e., $t_0 = x_0 + y_0$ with $x_0 \in K$ and
\begin{equation}\label{eq*}
\langle u,t_0 \rangle = c, \mbox{ i.e., } \langle u,x_0 \rangle = c- \langle u,y_0 \rangle\,.
\end{equation}
Pick $x^*, x^{**} \in K$ such that
$$
\langle u, x^*\rangle = \min \{\langle u,x\rangle : x \in K \}, \quad \langle u, x^{**}\rangle = \max \{\langle u,x\rangle : x \in K\}\,.
$$
For every $i = 1,2, \dots$, $H \cap (K+y_i) \not= \emptyset$ gives $\tilde{x}_i \in K$ such that
\begin{equation}\label{eq**}
 \langle u, \tilde{x}_i + y_i \rangle = c\,.
\end{equation}
We choose $t_i = x_i + y_i \in H \cap (K+y_i)$ as follows: We know from $\tilde{x}_i \in K$ that $\langle u, \tilde{x}_i \rangle \in [\langle u,x^*\rangle, \langle u,x^{**}\rangle]$.

\textbf{Case 1: $\langle u, \tilde{x}_i \rangle = \langle u, x_0 \rangle$. } In that case we put
\begin{equation}\label{(1)}
x_i : = x_0\,.
\end{equation}
Then
$t_i = x_0 + y_i \in K + y_i$ and $t_i \in H$, because $\langle u, t_i \rangle = \langle u, x_0 + y_i \rangle = \langle u, \tilde{x}_i + y_i \rangle \stackrel{\eqref{eq**}}{=} c$.

\textbf{Case 2: $\langle u, \tilde{x}_i \rangle \in [\langle u, x^* \rangle, \langle u, x_0 \rangle)$. } Then
\begin{equation}\label{(2)}
x_i : = \frac{\langle u, \tilde{x}_i \rangle - \langle u, x^*\rangle}{\langle u, x_0 \rangle - \langle u, x^*\rangle} x_0 +
\frac{\langle u, x_0 \rangle - \langle u, \tilde{x}_i \rangle}{\langle u,x_0 \rangle - \langle u, x^*\rangle} x^*
\end{equation}
satisfies $x_i \in [x_0, x^*] \subseteq K$, hence $t_i = x_i + y_i \in K + y_i$, and
$$
\langle u, x_i \rangle = \frac{\langle u, \tilde{x}_i \rangle - \langle u, x^*\rangle}{\langle u, x_0\rangle - \langle u, x^*\rangle} \langle u, x_0 \rangle +
\frac{\langle u, x_0\rangle - \langle u, \tilde{x}_i \rangle}{\langle u, x_0\rangle - \langle u, x^*\rangle} \langle u,x^*\rangle = \langle u, \tilde{x}_i \rangle\,.
$$
This gives $t_i \in H$, because
$\langle u, t_i \rangle = \langle u, x_i + y_i \rangle = \langle u, \tilde{x}_i + y_i \rangle \stackrel{\eqref{eq**}}{=} c$.

\textbf{Case 3: $\langle u, \tilde{x}_i \rangle \in (\langle u, x_0\rangle, \langle u, x^{**}\rangle]$. } Then
\begin{equation}\label{(3)}
x_i : = \frac{\langle u, x^{**}\rangle - \langle u, \tilde{x}_i \rangle}{\langle u, x^{**}\rangle - \langle u, x_0 \rangle} x_0 +
\frac{\langle u, \tilde{x}_i \rangle - \langle u, x_0\rangle}{\langle u, x^{**}\rangle - \langle u, x_0\rangle} x^{**}
\end{equation}
satisfies $x_i \in [x_0, x^{**}] \subseteq K$, hence $t_i = x_i + y_i \in K+y_i$, and
$$
\langle u, x_i \rangle = \frac{\langle u, x^{**}\rangle - \langle u, \tilde{x}_i \rangle}{\langle u, x^{**}\rangle - \langle u, x_0\rangle}
\langle u,x_0\rangle + \frac{\langle u, \tilde{x}_i \rangle - \langle u, x_0\rangle}{\langle u,x^{**}\rangle - \langle u,x_0\rangle}
\langle u, x^{**}\rangle = \langle u, \tilde{x}_i \rangle\,.
$$
This yields $t_i \in H$, because $\langle u, t_i \rangle = \langle u, x_i + y_i \rangle = \langle u, \tilde{x}_i + y_i \rangle \stackrel{\eqref{eq**}}{=} c$.

Finally, for proving $t_i \! \stackrel{{\scriptstyle i \to \infty}}{\longrightarrow} \! t_0$, we use the following arguments.
We get $x_i \! \stackrel{{\scriptstyle i \to \infty}}{\longrightarrow} \! x_0$ by partitioning the sequence $(x_i)_{i=1}^\infty$ into three subsequences corresponding to the above Cases 1--3, where each of the
subsequences (if it is infinite) converges to $x_0$. In Case 1 this is trivial, and in the other two cases it follows from
$$
\langle u, \tilde{x}_i \rangle \stackrel{\eqref{eq**}}{=} c - \langle u,y_i \rangle \! \stackrel{{\scriptstyle i \to \infty}}{\longrightarrow} \! c- \langle u, y_0\rangle \stackrel{\eqref{eq*}}{=} \langle u, x_0 \rangle\,
$$
and from the definitions (\ref{(2)}) and (\ref{(3)}).
This yields $t_i = x_i + y_i \! \stackrel{{\scriptstyle i \to \infty}}{\longrightarrow} \! x_0 + y_0 = t_0$.

\emph{Proof of (b).}  The inclusion $t_{i_j} \in H \cap (K+ y_{i_j})$ gives $x_{i_j} : = t_{i_j} - y_{i_j} \in K$. Hence
$x_{i_j} = t_{i_j} - y_{i_j} \! \stackrel{{\scriptstyle j \to \infty}}{\longrightarrow} \! t_0 - y_0 \in K$, because $K$ is closed. Thus
$t_0 \in K + y_0$. Moreover, $t_{i_j} \in H$ yields $t_{i_j} \! \stackrel{{\scriptstyle i \to \infty}}{\longrightarrow} \! t_0 \in H$, because $H$ is closed. Hence
$t_0 \in H \cap (K + y_0)$.
\end{proof}

\begin{remark}\label{rem1}
Note that $H$ cannot be replaced by an affine subspace $L$ of arbitrary dimension in Lemma~\ref{lemma1}. See the following example, where
$$
K  :=  \left\{(\xi_1,\xi_2,\xi_3) \in \bR^3 : |\xi_3| \le 1 - \sqrt{\xi^2_1 + \xi^2_2}\right\}\,.
$$
(Thus $K =  {\rm conv} (\{(\cos \varphi, \sin \varphi, 0) : 0 \le \varphi < 2\pi\} \cup \{(0,0, \pm 1)\})$ is a compact double cone.) Consider the affine subspace
$L :    =  \AFF(\{(1,0,0), (0,0,1)\})$ of $\bR^3$, and let
$y_i :  =  (1,0,0) - \left(\cos \frac{1}{i}, \sin \frac{1}{i},0\right) \! \stackrel{{\scriptstyle i \to \infty}}{\longrightarrow} \! y_0 = (0,0,0)$.
Then $L  \cap  (K + y_i) = \{(1,0,0)\}$ for any  $i$, and $L \cap  (K + y_0) = L \cap K = [(1,0,0), (0,0,1)]$. Hence
$L \cap (K + y_i)$ does not converge to $L \cap (K + y_0)$ in the described way.
\end{remark}

\begin{theorem}\label{theo1}
Let $C_1 \subseteq C_2 \subseteq \dots$ be an increasing sequence of b-convex bodies in a Minkowski space $(\bR^n, \| \cdot \|)$ such that
\begin{equation}\label{eq-assumption_bh}
\dim \left({\rm aff}
\left({\rm bh}_1 \left(\bigcup^\infty_{i=1}C_i\right)\right)\right) \in \{0,1, n-1,n\}\,.
\end{equation}
Then
\begin{equation}\label{(4)}
{\rm cl} \left(\bigcup^\infty_{i=1} C_i\right) =
{\rm bh}_1 \left(\bigcup^\infty_{i=1} C_i\right)\,.
\end{equation}
In particular, ${\rm cl}\left(\bigcup^\infty_{i=1} C_i\right)$
is a b-convex body.
\end{theorem}

\begin{proof}
Since ``$\subseteq$'' in (\ref{(4)}) is obvious, we prove now
``$\supseteq$''.

\textbf{Case 1:} For $\dim({\rm aff}({\rm bh}_1 (\bigcup^\infty_{i=1} C_i ))) = 0$,
$\bigcup^\infty_{i=1} C_i = \{x_0\}$ is a singleton. Thus
$C_i = \{x_0\}$ for all $i = 1,2,\dots$, and (\ref{(4)}) is trivial.

\textbf{Case 2:} For $\dim ({\rm aff}({\rm bh}_1 (\bigcup^\infty_{i=1} C_i))) = 1$,
bh$_1 (\bigcup^\infty_{i=1} C_i)$ is a line segment. Since every
closed line segment of the same direction and having smaller length is also b-convex,
each $C_i$ is a closed segment of that kind (perhaps of length $0$), $\bigcup^\infty_{i=1} C_i$ is a segment of that direction (not necessarily closed), and
$$
{\rm bh}_1 \left(\bigcup^\infty_{i=1} C_i\right) = {\rm  cl} \left(\bigcup^\infty_{i=1} C_i\right)\,.
$$

\textbf{Case 3: $\dim ({\rm aff}({\rm bh}_1 (\bigcup^\infty_{i=1} C_i))) \ge n-1 > 0$.}

Assume that we have ``$\not\supseteq$'' in (\ref{(4)}). Then there exists $x_0 \in {\rm bh}_1 (\bigcup^\infty_{i=1} C_i) \setminus {\rm cl}(\bigcup^\infty_{i=1} C_1)$, and we find
$\varepsilon_0 > 0$ such that $B(x_0,\varepsilon_0) \cap \bigcup^\infty_{i=1} C_i = \emptyset$. Consequently, there exists $x_1 \in {\rm relint} ({\rm bh}_1(\bigcup^\infty_{i=1} C_i))
\setminus (\bigcup^\infty_{i=1} C_i)$, i.e.,
\begin{equation}\label{(5)}
x_1 \in \, {\rm relint}\left({\rm bh}_1\left(\bigcup^\infty_{i=1}C_i\right)\right)
\end{equation}
and
\begin{equation}\label{(6)}
x_1 \notin C_i \, \mbox{ for } \, i = 1,2,\dots
\end{equation}

Property (\ref{(6)}) and Proposition~\ref{prop_separation}\ref{sep_b} give $y_i \in \bR^n$ such that
$$
x_1 \notin B(y_i,1) \, \mbox{ and } \, C_i \subseteq B(y_i,1) \, \mbox{ for } \, i = 1,2,\dots
$$
There exists a convergent subsequence $y_{i_j} \! \stackrel{{\scriptstyle j \to \infty}}{\longrightarrow} \! y_0 \in \bR^n$,
and, by continuity of the norm and the inclusions $C_1 \subseteq C_2 \subseteq \dots$, we have
\begin{equation}\label{(7)}
x_1 \notin {\rm int} (B(y_0,1)) \, \mbox{ and } \, \bigcup^\infty_{i=1} C_i \subseteq B(y_0,1)\,.
\end{equation}

\textbf{Subcase 3.1: $\dim ({\rm aff} ({\rm bh}_1 (\bigcup^\infty_{i=1} C_i))) = n$.}

With
(\ref{(7)}) we have ${\rm bh}_1 (\bigcup^\infty_{i=1} C_i) \subseteq B(y_0,1)$ and
$$
x_1 \notin {\rm int}\left({\rm bh}_1\left(\bigcup^\infty_{i=1} C_i\right)\right) \stackrel{{\rm (Subcase~3.1)}}{=} {\rm relint}\left({\rm bh}_1\left(\bigcup^\infty_{i=1} C_i\right)\right)\,,
$$
a contradiction to (\ref{(5)}).

\textbf{Subcase 3.2: $\dim ({\rm aff}({\rm bh}_1(\bigcup^\infty_{i=1} C_i))) = n-1$.}

Let $H := {\rm aff} ({\rm bh}_1 (\bigcup^\infty_{i=1} C_i))$.
From $y_{i_j} \! \stackrel{{\scriptstyle j \to \infty}}{\longrightarrow} \! y_0$ we get $B(y_{i_j},1) \! \stackrel{{\scriptstyle j \to \infty}}{\longrightarrow} \! B(y_0,1)$ in the Hausdorff metric, and with Lemma~\ref{lemma1} we get
$B(y_{i_j},1) \cap H \! \stackrel{{\scriptstyle j \to \infty}}{\longrightarrow} \!B(y_0,1) \cap H$ as well. Then, by $x_1 \notin B(y_{i_j},1)$, we obtain
$x_1 \notin {\rm int}_H(B(y_0,1) \cap H)$ (where ${\rm int}_H(\cdot)$ is the interior in the natural topology of $H$), whereas bh$_1 (\bigcup^\infty_{i=1} C_i)
\subseteq B(y_0,1) \cap H$ by (\ref{(7)}) and the choice of $H$. With this and the choice of $H$ we obtain
$x_1 \notin {\rm int}_H ({\rm bh}_1 (\bigcup^\infty_{i=1} C_i)) = {\rm relint } ({\rm bh}_1(\bigcup^\infty_{i=1} C_i))$, a contradiction to (\ref{(5)}). The
proof of ``$\supseteq$'' in (\ref{(4)}) is complete.

To show that cl$(\bigcup^\infty_{i=1} C_i) = {\rm bh}_1 (\bigcup^\infty_{i=1} C_i)$ is a b-convex body, it is enough to verify that
bh$_1 (\bigcup^\infty_{i=1} C_i) \not= \bR^n$, i.e., that $\bigcup^\infty_{i=1} C_i$ is contained in some ball of radius $1$. This is obvious by the second part of (\ref{(7)}) (which can be shown analogously in Cases 1 and 2).
\end{proof}

\begin{remark}\label{rem2}
Note that the technical assumption \eqref{eq-assumption_bh} is satisfied in each of the following situations:
\begin{enumerate}
\item[(i)] $n \le 3$,
\item[(ii)] $(\bR^n,\| \cdot \|)$ is strictly convex,
\item[(iii)] $\dim({\rm aff}(\bigcup^\infty_{i=1} C_i)) \ge n-1$ or, equivalently, $\dim({\rm aff}(C_{i_0})) \ge n-1$ for some $i_0$.
\end{enumerate}
\end{remark}

\begin{proof}
Situation (i) is trivial.
In situation (ii), Proposition 2 ((i) $\Leftrightarrow$ (xvii)) shows that $\operatorname{dim}(\AFF({\rm bh}_1(\bigcup^\infty_{i=1} C_i)))=n$ as soon as $\bigcup^\infty_{i=1} C_i$ is not a singleton.
Condition (iii) implies \eqref{eq-assumption_bh}, because $\bigcup^\infty_{i=1} C_i \subseteq {\rm bh}_1(\bigcup^\infty_{i=1} C_i)$.
\end{proof}

In Example~\ref{ex-dim4} we shall see that assumption \eqref{eq-assumption_bh} cannot be dropped in Theorem~\ref{theo1}.

\begin{theorem}\label{theo2}
Let $S \not= \emptyset$ be a subset of a Minkowski space $(\bR^n,  \| \cdot \|)$ such that
\begin{equation}\label{(9)}
\dim ({\rm aff} ({\rm bh}_1 (S))) \in \{0,1,n-1,n\}\,.
\end{equation}
Then
\begin{equation}\label{(11a)}
{\rm bh}_1 (S) = {\rm cl} \left(\bigcup_{F \subseteq S \, {\rm finite}} {\rm bh}_1 (F)\right)\,.
\end{equation}
\end{theorem}

\begin{proof}
The inclusion ``$\supseteq$'' is evident. For ``$\subseteq$'', let
$x_1, x_2,\ldots \in S$ be such that $S = {\rm cl} (\{x_1,x_2,\dots\})$.
Putting $C_i = {\rm bh}_1 (\{x_1,\dots, x_i\})$ for $i = 1,2,\dots$, we obtain $C_1 \subseteq C_2 \subseteq \dots$ If $C_{i_0}$ is not a b-convex body for some $i_0$, then ${\rm bh}_1(\{x_1,\dots,x_{i_0}\}) = C_{i_0} = \bR^n$, and (\ref{(11a)}) is obvious (both sides are $\bR^n$).
Hence we can assume that all $C_i$ are b-convex bodies. Moreover,
\begin{equation}\label{(11)}
\begin{array}{rl}
{\rm bh}_1 \left(\bigcup\limits^\infty_{i=1} C_i\right) \!\!\!& = \; {\rm bh}_1 \left(\bigcup\limits^\infty_{i=1} {\rm bh}_1(\{x_1,\dots,x_i\})\right) \;=\; {\rm bh}_1 \left(\bigcup\limits^\infty_{i=1} \{x_1,\dots,x_i\}\right)\\[3ex]
& = \; {\rm bh}_1 (\{x_1,x_2,\dots\}) \;=\; {\rm bh}_1 ({\rm cl}(\{x_1,x_2,\dots\})) \;=\; {\rm bh}_1(S)\,,
\end{array}
\end{equation}
which gives
$$
\dim\left({\rm aff}\left({\rm bh}_1\left(\bigcup^\infty_{i=1} C_i\right)\right)\right) = \dim ({\rm aff}({\rm bh}_1 (S))) \stackrel{\eqref{(9)}}{\in} \{0,1,n-1,n\}\,.
$$
Now we can apply Theorem \ref{theo1} and obtain
$$
\begin{array}{rcl}
{\rm bh}_1 (S) &\!\!\stackrel{\eqref{(11)}}{=}\!\!& {\rm bh}_1 \left(\bigcup \limits^\infty_{i=1} C_i\right) \;\stackrel{({\rm Theorem\; \ref{theo1}})}{=}\; {\rm cl} \left(\bigcup \limits^\infty_{i=1} C_i\right)
\;=\; {\rm cl} \left(\bigcup \limits^\infty_{i=1} {\rm bh}_1 (\{x_1,\dots,x_i\})\right)\\[3ex]
&\!\!\subseteq\!\!& {\rm cl} \left(\bigcup \limits_{F \subseteq S \,{\rm finite}} {\rm bh}_1 (F)\right)\,.
\end{array}
$$
\end{proof}

\begin{remark}\label{rem3}
As in Remark~\ref{rem2}, we see that (\ref{(9)}) holds in each of the following situations:
\begin{enumerate}
\item[(i)] $n \le 3$,
\item[(ii)] $(\bR^n, \| \cdot \|)$ is strictly convex,
\item[(iii)] $\dim({\rm aff}(S)) \ge n-1$.
\end{enumerate}
\end{remark}

The claim of Theorem~\ref{theo2} is shown in \cite[Theorem~1]{La-Na-Ta} under the stronger assumption that $\DIM(\AFF(\BH(S)))=n$. The authors ask in \cite[Problem~2.6]{La-Na-Ta} if the assumption can be dropped.
Our generalization to the additional cases $\DIM(\AFF(\BH(S))) \in \{0,1,n-1\}$ shows that the assumption from \cite{La-Na-Ta} can be weakened. However, the next example illustrates that the restrictions \eqref{eq-assumption_bh} in Theorem~\ref{theo1} and (\ref{(9)}) in Theorem~\ref{theo2} are essential, which shows that
the answer to Problem~2.6 from \cite{La-Na-Ta} is negative.

\begin{example}\label{ex-dim4}
We denote the Euclidean norm by $\|\cdot\|_2$ and consider convex bodies
\begin{eqnarray*}
K &=& \{(\kappa_1,\kappa_2,\kappa_3,0): \|(\kappa_1,\kappa_2)\|_2 \le 1,\|(\kappa_2,\kappa_3)\|_2 \le 1\},\\
L &=& \{(\lambda_1,0,\lambda_3,\lambda_4): \|(\lambda_1,\lambda_4)\|_2\le 1, |\lambda_3|\le 1\}
\end{eqnarray*}
in $\bR^4$. We define the unit ball $B(o,1)=B$ of a Minkowski space $(\bR^4,\|\cdot\|)$ by
\begin{eqnarray}
B &=& \CONV(K \cup L) \nonumber\\
&=& \{(\kappa_1+\lambda_1,\xi_2,\kappa_3+\lambda_3,\xi_4):\nonumber\\
&& \quad \max\{\|(\kappa_1,\xi_2)\|_2,\|(\xi_2,\kappa_3)\|_2\}+
\max\{\|(\lambda_1,\xi_4)\|_2, |\lambda_3|\} \le 1\}. \label{e-0}
\end{eqnarray}
For that space we shall see that
\begin{equation}\label{e-1}
\BH(\{(\alpha,0,0,0),(-\alpha,0,0,0)\})= [(\alpha,0,0,0),(-\alpha,0,0,0)] \;\mbox{ for }\; \alpha \in (0,1)
\end{equation}
and
\begin{equation}\label{e-2}
\BH(\{(1,0,0,0),(-1,0,0,0)\})= \{(\xi_1,0,0,\xi_4): \|(\xi_1,\xi_4)\|_2 \le 1\}.
\end{equation}

Consequently, the segments $C_i=\left[\left(1-\frac{1}{i},0,0,0\right),\left(-1+\frac{1}{i},0,0,0\right)\right]$, $i=1,2,\ldots$, form an increasing sequence of b-convex bodies, and we obtain
\begin{eqnarray*}
\CL\left(\bigcup_{i=1}^\infty C_i\right) &=& [(1,0,0,0),(-1,0,0,0)],\\
\BH\left(\bigcup_{i=1}^\infty C_i\right) &=& \{(\xi_1,0,0,\xi_4): \|(\xi_1,\xi_4)\|_2 \le 1\}.
\end{eqnarray*}
Hence \eqref{(4)} fails and $\CL(\bigcup_{i=1}^\infty C_i)$ is not b-convex.

Similarly, the relatively open segment $S=((1,0,0,0),(-1,0,0,0))$ satisfies
\begin{eqnarray*}
\BH(S) &=& \{(\xi_1,0,0,\xi_4): \|(\xi_1,\xi_4)\|_2 \le 1\},\\
\CL\left(\bigcup_{F \subseteq S \, {\rm finite}} {\rm bh}_1 (F)\right) &=& [(1,0,0,0),(-1,0,0,0)],
\end{eqnarray*}
and \eqref{(11a)} fails.
\end{example}

\begin{proof}[Proof of \eqref{e-1} and \eqref{e-2}]
\emph{Step 1. Verification of \eqref{e-1}. } Let $\alpha \in (0,1)$ be fixed. We use the linear functional
$$
f(\xi_1,\xi_2,\xi_3,\xi_4)= \sqrt{1-\alpha^2}\, \xi_2 + \alpha \xi_3=
\left\langle \left(\sqrt{1-\alpha^2},\alpha\right),(\xi_2,\xi_3) \right\rangle
$$
of Euclidean norm $\|f\|_2=\left\|\left(\sqrt{1-\alpha^2},\alpha\right)\right\|_2=1$,
and we define related level and sublevel sets $f_{=\mu}$, $f_{\le \mu}$, $f_{\ge \mu}$ by
$$
f_{=/\le/\ge \mu}= \{x \in \bR^4: f(x)=/\le/\ge \mu\}.
$$
We obtain, partially based on the Cauchy--Schwarz inequality,
$$
K \subseteq f_{\ge -1} \cap f_{\le 1}, \quad L \subseteq f_{\ge -\alpha} \cap f_{\le \alpha}, \quad L \cap f_{=1}= \emptyset.
$$
These yield
\begin{equation}\label{ep-1}
B \subseteq f_{\ge -1} \cap f_{\le 1}
\end{equation}
and
\begin{equation}\label{ep-2}
B \cap f_{=1}= K \cap f_{=1}= \left[\left(\alpha,\sqrt{1-\alpha^2},\alpha,0\right),\left(-\alpha,\sqrt{1-\alpha^2},\alpha,0\right)\right].
\end{equation}

Next note that
$\left(\pm \alpha,\pm \sqrt{1-\alpha^2}, \pm\alpha,0\right) \in K \subseteq B$ for arbitrary choice of signs. This implies $(\pm \alpha,0,0,0) \in B+\left(0,\pm\sqrt{1-\alpha^2},\pm\alpha,0\right)$, in particular
\begin{eqnarray*}
\{(\alpha,0,0,0),(-\alpha,0,0,0)\} \hspace{-15ex} && \\
&\subseteq&
B\left(\left(0,\sqrt{1-\alpha^2},\alpha,0\right),1\right) \cap B\left(\left(0,-\sqrt{1-\alpha^2},-\alpha,0\right),1\right),
\end{eqnarray*}
and in turn
\begin{eqnarray*}
\BH(\{(\alpha,0,0,0),(-\alpha,0,0,0)\}) \hspace{-31ex} && \\ &\subseteq&\!\!\!
B\left(\left(0,\sqrt{1-\alpha^2},\alpha,0\right),1\right) \cap B\left(\left(0,-\sqrt{1-\alpha^2},-\alpha,0\right),1\right)\\
&\stackrel{\eqref{ep-1}}{=}&\!\!\!
\left((B \cap f_{\ge -1})+\left(0,\sqrt{1-\alpha^2},\alpha,0\right)\right) \cap
\left((B \cap f_{\le 1})+\left(0,-\sqrt{1-\alpha^2},-\alpha,0\right)\right)\\
&=&\!\!\!
\left(B+\left(0,\sqrt{1-\alpha^2},\alpha,0\right)\right) \cap f_{\ge 0} \cap \left(B+\left(0,-\sqrt{1-\alpha^2},-\alpha,0\right)\right) \cap f_{\le 0}\\
&\subseteq&\!\!\! \left(B+\left(0,-\sqrt{1-\alpha^2},-\alpha,0\right)\right) \cap f_{= 0}\\
&=&\!\!\!
(B \cap f_{=1})+\left(0,-\sqrt{1-\alpha^2},-\alpha,0\right)\\
&\stackrel{\eqref{ep-2}}{=}&\!\!\! [(\alpha,0,0,0),(-\alpha,0,0,0)].
\end{eqnarray*}
This gives the inclusion ``$\subseteq$'' in \eqref{e-1}. The reverse inclusion is obvious, since the ball hull is closed and convex.

\emph{Step 2. Verification of the equivalence of
\begin{equation}\label{step2-left}
\{(1,0,0,0),(-1,0,0,0)\} \subseteq B(t,1)= B+t
\end{equation}
with
\begin{equation}\label{step2-right}
t=(0,0,\tau_3,0) \quad\mbox{ with }\quad \tau_3 \in [-1,1].
\end{equation}}

Suppose that $t=(\tau_1,\tau_2,\tau_3,\tau_4)$ satisfies \eqref{step2-left}. The inclusion $B = \CONV(K \cup L) \subseteq \CONV([-1,1]^4)=[-1,1]^4$ together with \eqref{step2-left} gives
\begin{equation}\label{step2-1}
\tau_1=0 \quad\mbox{ and }\quad \tau_3 \in [-1,1].
\end{equation}
Now the assumption $(-1,0,0,0) \in B+t$ amounts to $(-1,-\tau_2,-\tau_3,-\tau_4) \in B$ and, by symmetry, to $(1,\tau_2,\tau_3,\tau_4) \in B$. By \eqref{e-0}, this says that there are $\kappa_1,\lambda_1,\kappa_3,\lambda_3 \in \bR$ such that
$\kappa_1+\lambda_1=1$, $\kappa_3+\lambda_3=\tau_3$ and
\begin{equation}\label{step2-2}
\max\{\|(\kappa_1,\tau_2)\|_2,\|(\tau_2,\kappa_3)\|_2\}+
\max\{\|(\lambda_1,\tau_4)\|_2, |\lambda_3|\} \le 1.
\end{equation}
From $\kappa_1+\lambda_1=1$ we obtain
$$
1 \le |\kappa_1|+|\lambda_1|
\le \|(\kappa_1,\tau_2)\|_2+\|(\lambda_1,\tau_4)\|_2 \stackrel{\eqref{step2-2}}{\le} 1.
$$
Hence all inequalities in the last formula are identities and
$$
\tau_2=\tau_4=0.
$$
By \eqref{step2-1}, the implication ``\eqref{step2-left}$\Rightarrow$\eqref{step2-right}'' is proved.

The converse ``\eqref{step2-right}$\Rightarrow$\eqref{step2-left}''
amounts to $(\pm 1,0,-\tau_3,0) \in B$ for all $\tau_3 \in [-1,1]$. This is an obvious consequence of $(\pm 1,0,-\tau_3,0) \in L \subseteq B$.

Note that the equivalence ``\eqref{step2-left}$\Leftrightarrow$\eqref{step2-right}'' implies
\begin{equation}\label{step2-end}
\BH(\{(1,0,0,0),(-1,0,0,0)\})= \bigcap_{\tau_3 \in [-1,1]} (B+(0,0,\tau_3,0)).
\end{equation}

\emph{Step 3. Verification of ``$\subseteq$'' from \eqref{e-2}. }
Here the functional
$$
g(\xi_1,\xi_2,\xi_3,\xi_4)=\xi_3
$$
satisfies $K \cup L \subseteq g_{\ge -1} \cap g_{\le 1}$. Hence
\begin{equation}\label{step3-1}
B \subseteq g_{\ge -1} \cap g_{\le 1}
\end{equation}
and
\begin{eqnarray}
B \cap g_{=1} &=& \CONV((K \cap g_{=1}) \cup (L \cap g_{=1})) \nonumber\\
&=&\CONV(\{(\xi_1,0,1,0): |\xi_1|\le 1\} \cup \{\xi_1,0,1,\xi_4): \|(\xi_1,\xi_4)\|_2 \le 1\}) \nonumber\\
&=&\{\xi_1,0,1,\xi_4): \|(\xi_1,\xi_4)\|_2 \le 1\}. \label{step3-2}
\end{eqnarray}
Now we obtain the claim ``$\subseteq$'' from \eqref{e-2} by
\begin{eqnarray*}
\BH(\{(1,0,0,0),(-1,0,0,0)\}) \hspace{-28ex} && \\
&\stackrel{\eqref{step2-end}}{\subseteq}&
(B+(0,0,1,0)) \cap (B+(0,0,-1,0))\\
&\stackrel{\eqref{step3-1}}{=}&((B \cap g_{\ge -1})+(0,0,1,0)) \cap
((B \cap g_{\le 1})+(0,0,-1,0))\\
&=& (B+(0,0,1,0)) \cap g_{\ge 0} \cap (B+(0,0,-1,0)) \cap g_{\le 0}\\
&\subseteq& (B+(0,0,-1,0)) \cap g_{= 0}\\
&=& (B \cap g_{=1})+(0,0,-1,0)\\
&\stackrel{\eqref{step3-2}}{=}& \{(\xi_1,0,0,\xi_4): \|(\xi_1,\xi_4)\|_2 \le 1\}.
\end{eqnarray*}

\emph{Step 4. Verification of ``$\supseteq$'' from \eqref{e-2}. }
Let $\xi_1,\xi_4,\tau_3 \in \bR$ be such that $\|(\xi_1,\xi_4)\|_2 \le 1$ and $|\tau_3|\le 1$. Then $(\xi_1,0,-\tau_3,\xi_4) \in L \subseteq B$. Thus
$$
(\xi_1,0,0,\xi_4) \in B+(0,0,\tau_3,0) \quad\mbox{ if }\quad \|(\xi_1,\xi_4)\|_2 \le 1,\,|\tau_3|\le 1.
$$
By \eqref{step2-end}, this implies ``$\supseteq$'' from \eqref{e-2}.
\end{proof}


\section{Minimal representation of ball convex bodies as ball hulls}

In this section we will present, as announced, minimal representations of ball convex bodies in terms oft their ball exposed faces.

\begin{theorem}
\label{thm_meeting_F}
Let $K$ be a b-bounded b-convex body in a Minkowski space $(\mathbb{R}^n,\|\cdot\|)$ and let $S \subseteq K$. Then $\BH(S)=K$ if and only if every exposed b-face of $K$ meets $\CL(S)$.
\end{theorem}

\begin{proof}
For the proof of ``$\Rightarrow$'', suppose that there is an exposed b-face $F$ of $K$ such that $F \cap \CL(S)= \emptyset$. We have to show that $\BH(S) \ne K$. The b-face $F$ has a representation $F= K \cap S(y,1)$, where $S(y,1)$ is a supporting sphere of $K$. By $F \cap \CL(S)= \emptyset$, we obtain $\CL(S) \subseteq \INT(B(y,1))$ and, by Lemma~\ref{lem_basics}\ref{rad_interior}, $\CL(S) \subseteq B(y,r)$ for some $r < 1$. We fix $x_0 \in F$. Then $\|x_0-y\|=1$,
because $F \subseteq S(y,1)$, and $x_0 \notin B(y-(1-r)(x_0-y),1)$, since $\|x_0-(y-(1-r)(x_0-y))\|=(2-r)\|x_0-y\| > 1$. But $S \subseteq B(y,r) \subseteq B(y-(1-r)(x_0-y),1)$ by the triangle inequality.
Thus
$$
x_0 \notin B(y-(1-r)(x_0-y),1) \supseteq \BH(S) \quad\mbox{ and }\quad x_0 \in F \subseteq K,
$$
showing that $\BH(S) \ne K$.

For the converse implication ``$\Leftarrow$'', we suppose that $\BH(S) \ne K$ and will show that $\CL(S)$ misses at least one exposed b-face $F_0$ of $K$. Since $\BH(S) \ne K$ and $\BH(S) \subseteq K$ by Lemma~\ref{lem_basics}, there is $x_0 \in K \setminus \BH(S)$. By Proposition~\ref{prop_separation}\ref{sep_b}, we can separate $x_0$ from the b-convex body $\BH(S) \subseteq K$ by a sphere $S(y_0,1)$,
\begin{equation}
\label{eq_t1_1}
\CL(S) \subseteq \BH(S) \subseteq B(y_0,1) \quad\mbox{ and }\quad x_0 \notin B(y_0,1).
\end{equation}
The b-boundedness of $K$ gives $y_1 \in \mathbb{R}^n$ such that
\begin{equation}
\label{eq_t1_2}
\CL(S) \subseteq K \subseteq \INT(B(y_1,1)).
\end{equation}
We consider the balls $B_\lambda:=B(y_0+\lambda(y_1-y_0),1)$ for $\lambda \in [0,1]$. Then $K \not\subseteq B_0$ by \eqref{eq_t1_1}
and $K \subseteq B_1$ by \eqref{eq_t1_2}. Consequently, there exists
$$
\lambda_0= \min\{\lambda \in [0,1]: K \subseteq B_\lambda\} \in (0,1].
$$
(It is a consequence of the continuity of $\|\cdot\|$ that $\lambda_0$ is really attained as a minimum.)
By the definition of $\lambda_0$ and a compactness argument, the set $F_0=K \cap \BD(B_{\lambda_0})$ is non-empty, so that $S_{\lambda_0}:=\BD(B_{\lambda_0})$ is a supporting sphere of $K$ and $F_0$ is an exposed b-face.

Now it remains to show that $F_0 \cap \CL(S)= \emptyset$. Suppose that this is not the case; i.e., there exists $z_0 \in F_0 \cap \CL(S)$.
The inclusions \eqref{eq_t1_1} and \eqref{eq_t1_2} yield $\|z_0-y_0\| \le 1$ and $\|z_0-y_1\| < 1$. Finally, the inclusion $z_0 \in F_0 \subseteq S_{\lambda_0}=S(y_0+\lambda_0(y_1-y_0),1)$ gives
\begin{eqnarray*}
1 &=& \|z_0-(y_0+\lambda_0(y_1-y_0))\| \\
&=& \|\lambda_0(z_0-y_1)+(1-\lambda_0)(z_0-y_0)\| \\
&\le& \lambda_0 \|z_0-y_1\|+ (1-\lambda_0)\|z_0-y_0\| \\
&<& \lambda_0+(1-\lambda_0) \\
&=& 1.
\end{eqnarray*}
This contradiction completes the proof.
\end{proof}

Note that the proof of ``$\Rightarrow$'' did not require b-boundedness of $K$. However, b-boundedness is essential for ``$\Leftarrow$''. To see this, consider a closed ball $K=B(y,1)$
of radius $1$. (Proposition~\ref{prop_str_conv}\ref{p_i}$\Rightarrow$\ref{p_vi} says that these are the only b-convex bodies that are not b-bounded, provided that the norm $\|\cdot\|$ is strictly convex.) Then the only supporting sphere of $K$ is $S(y,1)$, and the only exposed b-face is $F=K \cap S(y,1)= S(y,1)$. Then every singleton $S=\{x_0\} \subseteq S(y,1)$ satisfies the condition from Theorem~\ref{thm_meeting_F}, but $\BH(S)=\{x_0\}$ is not $K$.

\begin{example}
\label{ex_linfty}
Consider the space $l_\infty^2=(\mathbb{R}^2,\|\cdot\|_\infty)$ with unit ball $[-1,1]^2$. Then all b-convex bodies are of the form $[\alpha_1,\beta_1] \times [\alpha_2,\beta_2]$ with $0 \le \beta_i-\alpha_i \le 2$, $i=1,2$. We restrict our consideration to b-bounded b-convex bodies $K$ with non-empty interior. These are rectangles $K=[\alpha_1,\beta_1] \times [\alpha_2,\beta_2]$ with $0< \beta_i-\alpha_i < 2$, $i=1,2$. The exposed b-faces of $K$ are the edges
$$
F_1=[\alpha_1,\beta_1] \times \{\beta_2\}, F_2=\{\alpha_1\} \times [\alpha_2,\beta_2], F_3=[\alpha_1,\beta_1] \times \{\alpha_2\}, F_2=\{\beta_1\} \times [\alpha_2,\beta_2]
$$
and the unions $F_1 \cup F_2$, $F_2 \cup F_3$, $F_3 \cup F_4$, $F_4 \cup F_1$. Theorem~\ref{thm_meeting_F} says that a set $S \subseteq K$ satisfies $\BH(S)=K$ if and only if $\CL(S) \cap F_j \ne \emptyset$ for $j=1,2,3,4$. Consequently, when searching for
minimal sets $S$ (under inclusion) with $\BH(S)=K$, we need to
find a minimal set $S$ containing at least one point from each of $F_1,F_2,F_3,F_4$. Such $S$ may consist of $2$ (if $S$ is composed of two vertices symmetric with respect to the center of $K$), $3$ or $4$ points (if $S$ contains exactly one point from the relative interior of each $F_i$).

This example can be generalized for boxes in $l_\infty^n$, $n \ge 1$. Corresponding minimal sets must contain a point in every (classical) facet of a box and may consist of $2,\ldots,2n$ elements.
\end{example}

\begin{corollary}
\label{cor_necessary}
If a subset $S$ of a b-bounded b-convex body $K$ in a Minkowski space $(\mathbb{R}^n,\|\cdot\|)$ satisfies $\BH(S)=K$, then $\BEXP(K) \subseteq \CL(S)$.
\end{corollary}

\begin{proof}
If $x \in \BEXP(K)$, then $\{x\}$ is an exposed b-face of $K$. Now Theorem~\ref{thm_meeting_F} yields $\{x\} \cap \CL(S) \ne \emptyset$; i.e., $x \in \CL(S)$.
\end{proof}

\begin{theorem}
\label{thm_bh_bexp}
A subset $S$ of a b-bounded b-convex body $K$ in a strictly convex Minkowski space $(\mathbb{R}^n,\|\cdot\|)$ satisfies $\BH(S)=K$ if and only if $\BEXP(K) \subseteq \CL(S)$. In particular,
$$
K=\BH(\BEXP(K))\,,
$$
and $\CL(\BEXP(K))$ is the unique minimal (under inclusion) closed subset of $\mathbb{R}^n$ whose ball hull is $K$.
\end{theorem}

\begin{proof}
The implication ``$\Rightarrow$'' of ``$\BH(S)=K \Leftrightarrow \BEXP(K) \subseteq \CL(S)$'' is given by Corollary~\ref{cor_necessary}.
To see ``$\Leftarrow$'', it is enough to show that
\begin{equation}
\label{eq_t2_goal}
K \subseteq \BH(\BEXP(K)).
\end{equation}
Indeed, if $\BEXP(K) \subseteq \CL(S)$ and if \eqref{eq_t2_goal} is verified, then Lemma~\ref{lem_basics}\ref{bh_inclusion} and \ref{bh_bhbh} give
$$
K \subseteq \BH(\BEXP(K)) \subseteq \BH(\CL(S)) = \BH(S) \subseteq \BH(K) = K,
$$
and ``$\Leftarrow$'' is proved.

Now, for showing \eqref{eq_t2_goal}, let us assume that there exists $x_0 \in K \setminus \BH(\BEXP(K))$. The b-boundedness of $K$ implies b-boundedness of the subset $\tilde{K}:=\BH(\BEXP(K)) \subseteq \BH(K)=K$. Separation of $x_0$ from $\tilde{K}$ by Proposition~\ref{prop_separation}\ref{sep_c} and the b-boundedness of $K$ yield the existence of $y_0,y_1 \in \mathbb{R}^n$ and $r \in (0,1)$ such that
\begin{eqnarray}
& \tilde{K} \subseteq B(y_0,r) \quad\mbox{ and }\quad x_0 \notin B(y_0,1), & \label{eq_t2_sep} \\
& \tilde{K} \subseteq K \subseteq \INT(B(y_1,r)). & \label{eq_t2_cover}
\end{eqnarray}
Similarly as in the proof of Theorem~\ref{thm_meeting_F}, we define $B_\lambda^r:= B(y_0+\lambda(y_1-y_0),r)$ for $\lambda \in [0,1]$ and, exploiting $K \not\subseteq B_0^r$ from \eqref{eq_t2_sep} and $K \subseteq B_1^r$ from \eqref{eq_t2_cover}, find
\begin{equation}
\label{eq_t2_lambda0}
\lambda_0=\min\{\lambda \in [0,1]: K \subseteq B_\lambda^r\} \in (0,1].
\end{equation}
Then there exists $x_1 \in K \cap S^r_{\lambda_0}$, where $S_{\lambda_0}^r:=\BD\left(B_{\lambda_0}^r\right)$.

Next we show that
\begin{equation}
\label{eq_t2_x1_notin}
x_1 \notin \tilde{K}.
\end{equation}
Indeed, if $x_1$ belonged to $\tilde{K}$, we had $\|x_1-y_0\| \le r$ and $\|x_1-y_1\|< r$. The inclusion $x_1 \in S_{\lambda_0}^r=S(y_0+\lambda_0(y_1-y_0),r)$ gave
\begin{eqnarray*}
r
&=&\|x_1-(y_0+\lambda_0(y_1-y_0))\| \\
&=& \|\lambda_0(x_1-y_1)+(1-\lambda_0)(x_1-y_0)\| \\
&\le& \lambda_0 \|x_1-y_1\|+ (1-\lambda_0)\|x_1-y_0\| \\
&<& \lambda_0 r+(1-\lambda_0)r \\
&=& r.
\end{eqnarray*}
This contradiction proves \eqref{eq_t2_x1_notin}.

Since $B_{\lambda_0}^r$ is a b-convex body by Lemma~\ref{lem_basics}\ref{ball_is_body} and since $x_1 \in S_{\lambda_0}^r=\BD(B_{\lambda_0}^r)$, Proposition~\ref{prop_separation}\ref{sep_a} gives $y_2 \in \mathbb{R}^n$ such that
$$
B_{\lambda_0}^r \subseteq B(y_2,1) \quad\mbox{ and }\quad x_1 \in B_{\lambda_0}^r \cap S(y_2,1).
$$
Proposition~\ref{prop_str_conv}\ref{p_i}$\Rightarrow$\ref{p_xi} tells us that
$$
B_{\lambda_0}^r \cap S(y_2,1)=\{x_1\},
$$
because $\|\cdot\|$ is strictly convex. Using the known inclusions $x_1 \in K$ and $K \subseteq B_{\lambda_0}^r$, we get
$$
K \subseteq B(y_2,1) \quad\mbox{ and }\quad K \cap S(y_2,1)=\{x_1\}.
$$
Hence $x_1 \in \BEXP(K)$. But \eqref{eq_t2_x1_notin} says that $x_1 \notin \BH(\BEXP(K))$. This final contradiction establishes \eqref{eq_t2_goal} and completes the proof.
\end{proof}

Theorem~\ref{thm_bh_bexp} says in particular that every b-bounded b-convex body in a strictly convex Minkowski space gives rise to a unique minimal closed subset whose ball hull is that body. We have seen in Example~\ref{ex_linfty} that this is not necessarily the case if the norm fails to be strictly convex.

\begin{example}
The set $\BEXP(K)$ of all b-exposed points of a b-bounded b-convex body is not necessarily closed. An example of that kind in the Euclidean plane is the convex disc $K$ bounded by the arcs
$$
\begin{array}{l}
\Gamma_1 = \left\{\left(\cos(\varphi),-\frac{\sqrt{3}}{4}+\sin(\varphi)\right): \frac{\pi}{3} \le \varphi \le \frac{2\pi}{3} \right\},\\[1.5ex]
\Gamma_2 = \left\{\left(\cos(\varphi),\frac{\sqrt{3}}{4}+\sin(\varphi)\right): \frac{4\pi}{3} \le \varphi \le \frac{5\pi}{3} \right\},\\[2ex]
\Gamma_3 = \left\{\left(\frac{1}{4}+\frac{1}{2}\cos(\varphi),\frac{1}{2}\sin(\varphi)\right): -\frac{\pi}{3} < \varphi < \frac{\pi}{3} \right\},\\[2ex]
\Gamma_4 = \left\{\left(-\frac{1}{4}+\frac{1}{2}\cos(\varphi),\frac{1}{2}\sin(\varphi)\right): \frac{2\pi}{3} < \varphi < \frac{4\pi}{3} \right\}.
\end{array}
$$
Thus, $K$ is a b-bounded b-convex body with $\BEXP(K)=\Gamma_3 \cup \Gamma_4$. Exposed b-faces that are no singletons are $\Gamma_1$ and $\Gamma_2$.
\end{example}

\begin{corollary}
\label{cor_t2_1}
If $K$ is a b-bounded b-convex body in a strictly convex Minkowski space $(\mathbb{R}^n,\|\cdot\|)$, then every exposed b-face of $K$ meets the closure of $\BEXP(K)$.
\end{corollary}

\begin{proof}
This is a consequence of Theorems~\ref{thm_meeting_F} and \ref{thm_bh_bexp}.
\end{proof}

\begin{corollary}
\label{cor_t2_2}
If $S$ is a non-empty b-bounded subset of a strictly convex Minkowski space $(\mathbb{R}^n,\|\cdot\|)$, then $\BEXP(\BH(S)) \subseteq \CL(S)$.
\end{corollary}

\begin{proof}
By Lemma~\ref{lem_basics}\ref{rad_bh} and \ref{bh_bhbh}, $K:=\BH(S)$ is a b-bounded b-convex body and $S$ is a subset of $K$. Now Theorem~\ref{thm_bh_bexp} says that $\CL(S) \supseteq \BEXP(K)=\BEXP(\BH(S))$.
\end{proof}

Example~\ref{ex_linfty} gives b-bounded b-convex bodies not having any b-exposed points. This shows that Theorem~\ref{thm_bh_bexp} and Corollary~\ref{cor_t2_1} fail in general if the underlying norm is not strictly convex.

A similar reason justifies the assumption of b-boundedness in Theorem~\ref{thm_bh_bexp} and Corollary~\ref{cor_t2_1}:

\begin{proposition}
\label{prop_bexp_empty}
If a b-convex body $K$ in a Minkowski space $(\mathbb{R}^n,\|\cdot\|)$ is not b-bounded, then $\BEXP(K)=\emptyset$.
\end{proposition}

\begin{proof}
We have $\RAD(K)=1$, because $K$ is not b-bounded. Hence every supporting sphere $S(x,1)$ of $K$ is the boundary of a circumball $B(x,1)$. By Lemma~\ref{lem_circumintersection}, $|K \cap S(x,1)| \ge 2$. Hence none
of the exposed b-faces of $K$ is a singleton and $K$ has no b-exposed points.
\end{proof}


\section{An application to diametrically maximal sets}

A bounded non-empty set $C \subseteq \mathbb{R}^n$ is called \emph{complete} (or \emph{diametrically maximal}) if $\DIAM(C \cup \{x\})> \DIAM(C)$ for every $x \in \mathbb{R}^n \setminus C$; see \cite{meissner1911, jessen1929, eggleston1965, groemer1986}. Complete sets are necessarily convex bodies, and in the Euclidean case or for $n=2$ any complete set is of constant width. A complete set $C$ is called a \emph{completion} of a bounded non-empty set $S$ if $S \subseteq C$ and $\DIAM(C)=\DIAM(S)$. Zorn's lemma shows that every bounded non-empty subset of $\mathbb{R}^n$ has at least one completion. In $n$-dimensional
Minkowski spaces $(n \ge 3)$, the family of complete bodies can form a much richer class than that of bodies of constant width; see \cite{Mo-Sch1, Mo-Sch2} for recent contributions.

The following problem was posed in \cite[Section~4]{martini_et_al_2014}: \emph{Given a complete set $C \subseteq \mathbb{R}^n$, find all convex bodies $K_0 \subseteq C$ such that $C$ is the unique completion of $K_0$ and, moreover, there is no convex body $K \subseteq K_0$, $K \ne K_0$, such that $C$ is the unique completion of $K$.}

Without loss of generality, we can assume that $\DIAM(C)=1$. The following lemma summarizes particular relevant statements from the literature (for \ref{eggleston*} and \ref{eggleston}, see \cite[Section 1(E)]{eggleston1965}; for \ref{groemer} and \ref{groemer_proof}, see \cite[Theorem~5]{groemer1986} and the short proof given there).

\begin{lemma}
\label{lem_complete}
The following are satisfied in every Minkowski space $(\mathbb{R}^n,\|\cdot\|)$:
\begin{enumerate}[label={(\alph*)}]
\item A set $C \subseteq \mathbb{R}^n$ of diameter $1$ is complete if and only if, for every $x_1 \in \BD(C)$, there exists $x_2 \in \BD(C)$ such that $\|x_1-x_2\|=1$. \label{eggleston*}
\item A set $C \subseteq \mathbb{R}^n$ of diameter $1$ is complete if and only if $C=\bigcap_{x \in C} B(x,1)$. \label{eggleston}
\item A set $S \subseteq \mathbb{R}^n$ of diameter $1$ has a unique completion if and only if $\BH(S)$ is complete. \label{groemer}
\item If a set $S \subseteq \mathbb{R}^n$ of diameter $1$ has a unique completion $C$, then $C=\BH(S)$. \label{groemer_proof}
\end{enumerate}
\end{lemma}

\begin{proposition}
\label{prop_dm}
Let $C$ be a complete set of diameter $1$ in a Minkowski space $(\mathbb{R}^n,\|\cdot\|)$, and let $K \subseteq C$ be a convex body. The following three conditions are equivalent:
\begin{enumerate}[label={(\Roman*)}]
\item $C$ is the unique completion of $K$. \label{dm_1}
\item $\BH(K)=C$. \label{dm_2}
\item $K$ meets every exposed b-face of $C$. \label{dm_3}
\end{enumerate}
If, in addition, $\|\cdot\|$ is strictly convex, then \ref{dm_1}, \ref{dm_2}, and \ref{dm_3} are equivalent to
\begin{enumerate}[label={(\Roman*)}]
\setcounter{enumi}{3}
\item $\CL(\CONV(\BEXP(C))) \subseteq K$. \label{dm_4}
\end{enumerate}
\end{proposition}

\begin{proof}
First note that $C$ is b-bounded by (\ref{(1a)}), because $\DIAM(C)=1$, and Lemma~\ref{lem_complete}\ref{eggleston} shows that $C$ is a b-bounded b-convex body.

\ref{dm_1}$\Rightarrow$\ref{dm_2}: Since $C$ is a completion of $K$, we obtain $\DIAM(K)=\DIAM(C)=1$. Now Lemma~\ref{lem_complete}\ref{groemer_proof} gives \ref{dm_1}$\Rightarrow$\ref{dm_2}.

\ref{dm_2}$\Leftrightarrow$\ref{dm_3} and \ref{dm_2}$\Leftrightarrow$\ref{dm_4} follow from Theorems~\ref{thm_meeting_F} and \ref{thm_bh_bexp}, respectively.

(\ref{dm_2}$\wedge$\ref{dm_3})$\Rightarrow$\ref{dm_1}: By \ref{dm_3}, there exists $x_1 \in K \cap \BD(C)$. Lemma~\ref{lem_complete}\ref{eggleston*} gives $x'_2 \in \BD(C)$ such that $x'_2 \in S(x_1,1)$. Then $S(x_1,1)$ is a supporting sphere of $C$ and $F=C \cap S(x_1,1)$ is an exposed b-face of $C$. By condition \ref{dm_3}, there exists $x_2 \in K \cap F \subseteq K \cap S(x_1,1)$. We obtain $\DIAM(K)=1$, because
$$
1=\|x_1-x_2\| \le \DIAM(K) \le \DIAM(C)=1.
$$
Now Lemma~\ref{lem_complete}\ref{groemer} and \ref{groemer_proof} gives \ref{dm_2}$\Rightarrow$\ref{dm_1}, and we are done.
\end{proof}

Criteria \ref{dm_3} and \ref{dm_4} from Proposition~\ref{prop_dm} help to characterize minimal convex bodies $K_0$ in a complete set $C$ such that $C$ is the unique completion of $K_0$.

\begin{example}
We consider the space $l_\infty^n$ as in Example~\ref{ex_linfty}. The only complete sets in that space are closed balls (see \cite[Corollary 2]{eggleston1965}), so that a complete set $C$ of diameter $1$ is necessarily a box (i.e., a square if $n=2$) with edges of length $1$ parallel to the coordinate axes. The equivalence of \ref{dm_1} and \ref{dm_3} in Proposition~\ref{prop_dm} says that $C$ is the unique completion of a convex body $K \subseteq C$ if and only if $K$ meets each of the $2n$ facets of $C$. It is easy to find minimal convex bodies $K_0$ with that property: Such $K_0$ is the convex hull of a minimal set $S$ consisting of at least one point from each facet of $C$. If $n=2$, then such $K_0$ can be a line segment (a diagonal of $C$), a triangle or a quadrangle. For arbitrary $n \ge 2$,
the number of vertices of such $K_0$ can be $2$ (if $K_0$ is a diagonal of $C$ passing through the center of $C$), $3$, \ldots, $2n$ (e.g., if $K_0$ is the cross polytope generated by the centers of the $2n$ facets of $C$).
\end{example}

If the underlying Minkowski space is strictly convex, then Proposition~\ref{prop_dm} shows that the problem mentioned above has a unique solution.

\begin{corollary}
Let $C$ be a complete set of diameter $1$ in a strictly convex Minkowski space $(\mathbb{R}^n,\|\cdot\|)$. Then $K_0=\CL(\CONV(\BEXP(C)))$ is the unique minimal (under inclusion) convex body whose unique completion is $C$.
\end{corollary}



\section{Open questions}

\subsection{Spindle convexity in Minkowski spaces}
In \cite{La-Na-Ta} $\BH(\{x_1,x_2\})$ is called the \emph{spindle} of $x_1,x_2 \in \mathbb{R}^n$, which generalizes the corresponding notion from Euclidean space (see, e.g.,\ \cite{bezdek_et_al_2007} and the references given there). A set $S \subseteq \mathbb{R}^n$ is called \emph{spindle convex} if, for all $x_1,x_2 \in S$, $S$ covers the whole spindle of $x_1$ and $x_2$. This gives rise to the concept of the \emph{spindle convex hull} of a subset of $\mathbb{R}^n$.
Note that spindle convex sets are not necessarily closed, in contrast to b-convex sets. Closed sets turn out to be spindle convex if and only if they are b-convex, provided the underlying Minkowski space is Euclidean or two-dimensional or its unit ball is (an affine image of) a cube (see \cite[Corollary~3.4]{bezdek_et_al_2007}, \cite[Corollaries 3.13 and 3.15]{La-Na-Ta}). An example in (an affine image of) the space $l_1^3$ from Example~\ref{ex-l1} shows that closed spindle convex sets need not be b-convex in general (see \cite[Example~3.1]{La-Na-Ta}).

We define a related hierarchy of notions of convexity by calling a set $S \subseteq \bR^n$ \emph{$k$-spindle convex}, $k \in \{2,3,\ldots\}$, if $\BH(\{x_1,\ldots,x_k\}) \subseteq S$ for all $x_1,\ldots,x_k \in S$. We call $S$ \emph{$\ast$-spindle convex} if $\BH(F) \subseteq S$ for every finite $F \subseteq S$ (i.e., if $S$ is $k$-spindle convex for all $k=2,3,\ldots$).

Are the $k$-spindle convex hulls and the $\ast$-spindle convex hull of a closed set closed? Clearly, every b-convex set is $\ast$-spindle convex. Is every closed $\ast$-spindle convex set b-convex? Theorem~\ref{theo2} says that in many situations $\BH(S)$ is the closure of the $\ast$-spindle convex hull of $S$. On the other hand, the relatively open segment $S$ from Example~\ref{ex-dim4} is $\ast$-spindle convex, but $\CL(S)$ is
not even $2$-spindle convex. Given an arbitrary Minkowski space $(\bR^n,\|\cdot\|)$, does there exist $k \in \{2,3,\ldots\}$ such that $\ast$-spindle convexity coincides with $k$-spindle convexity? Given $k \in \{2,3,\ldots\}$, does there exist a Minkowski space $(\bR^n,\|\cdot\|)$ such that $k$-spindle convexity differs from $(k+1)$-spindle convexity?
These and related questions might be studied to continue naturally our investigations here.


\subsection{Generalized Minkowski spaces}

Our results are shown in the framework of a Minkowski space. What remains true if the norm is replaced by a gauge, i.e., if the unit ball is no longer necessarily centered at $o$?

\subsection{M\"obius geometry}

One might check whether there are interesting connections (e.g., regarding the used methods and tools) to M\"obius geometry where spheres also play somehow the role of hyperplanes;
see, e.g., \cite{Vo, L-T}.



\end{document}